\documentclass[10pt]{article}
\usepackage{amsmath}
\usepackage{amsfonts}
\usepackage{gensymb}
\usepackage{wasysym}
\usepackage{amsthm}
\usepackage{enumerate}
\usepackage{amssymb}
\usepackage{tikz}

\newcommand{\fr}{^{\frown}}
\newcommand{\pz}{\Pi^0_1}
\newcommand{\dt}{\Delta^0_2}
\newcommand{\ov}{\overrightarrow}

\newcommand{\Z}{\mathbb{Z}}
\newcommand{\A}{{\mathcal A}}
\newcommand{\B}{{\mathcal B}}

\newcommand{\G}{{\mathcal G}} 
\newcommand{\F}{{\mathcal F}} 
\newcommand{\overbar}{\overline}

\newcommand{\N}{\mathbb N}

\newcommand{\restr}{\upharpoonright}  

\newcommand{\K}{{\mathcal K}} 
\newcommand{\HH}{{\mathcal H}}
\newcommand{\Or}{{\mathcal O}}
\newcommand{\Q}{{\mathbb Q}}

\newcommand{\la}{\langle}

\newcommand{\ra}{\rangle}

\newtheorem{theorem}{Theorem}[section]
\newtheorem{defn}[theorem]{Definition}
\newtheorem{lemma}[theorem]{Lemma}

\newtheorem{corollary}[theorem]{Corollary}
\newtheorem{cor}[theorem]{Corollary}
\newtheorem{example}{Example}
\newtheorem{proposition}[theorem]{Proposition}
\newtheorem{prop}[theorem]{Proposition}

\begin{document}

\pagestyle{headings}

\title{Computability and Categoricity of Weakly Ultrahomogeneous Structures}

\author{Francis Adams and Douglas Cenzer \\
Department of Mathematics, University of Florida \\
P.O. Box 118105,  Gainesville, Florida 32611 \\
fsadams@ufl.edu,  cenzer@math.ufl.edu}

%

\date{}
\maketitle

\begin{abstract}
This paper investigates the effective categoricity of ultrahomogeneous
structures. It is shown that any computable ultrahomogeneous structure
is $\Delta^0_2$ categorical. A structure $\A$ is said to be \emph{weakly
  ultrahomogeneous} if there is a finite (\emph{exceptional}) set of
elements $a_1,\dots,a_n$ such that $\A$ becomes ultrahomogeneous when
constants representing these elements are added to the
language. Characterizations are obtained for  weakly
ultrahomogeneous linear orderings, equivalence structures,
injection structures and trees, and these are compared with characterizations of the
computably categorical and $\Delta^0_2$ categorical structures. Index sets are used to
determine the complexity of the notions of ultrahomegenous and weakly ultrahomogeneous 
for various families of structures. 

\end{abstract}



\section{Introduction}

Computable model theory studies the algorithmic properties of effective
mathematical structures and the relationships among such structures.
The effective categoricity of a computable structure $\A$ measures the
possible complexity of isomorphisms between $\A$ and computable copies
of $\A$, and is an important gauge of the complexity of $\A$. 

This paper will only be concerned with countable structures, countably infinite unless otherwise stated. We say a countable structure (model) $\A$ is computable if its
universe $A$ is computable and all of its functions and relations are
uniformly computable.  Given two computable structures, we will say
they are computably isomorphic if there exists an isomorphism between
them that is computable.  For a single computable structure
$\mathcal{A}$, we will say $\mathcal{A}$ is computably categorical if
every computable structure isomorphic to $\mathcal{A}$ is in fact
computably isomorphic to $\mathcal{A}$.  More generally, we will say
two computable structures are $\Delta^0_{\alpha}$ isomorphic if there
exists an isomorphism between them that is $\Delta^0_{\alpha}$ and we
will say a computable structure $\mathcal{A}$ is $\Delta^0_{\alpha}$
categorical if every computable structure isomorphic to $\mathcal{A}$
is $\Delta^0_{\alpha}$ isomorphic to $\mathcal{A}$ and 
will say that a computable structure $\mathcal{A}$ is \emph{relatively} $\Delta^0_{\alpha}$
categorical if, for  every structure $\B$  isomorphic to $\mathcal{A}$, there 
is an isomorphism which is $\Delta^0_{\alpha}$ relative to the diagram of $\B$.  
 More generally, an arbitrary structure $\A$  is relatively $\Delta^0_{\alpha}$ categorical 
if for any structure $\B$ isomorphic to $\A$ there is an isomorphism which is $\Delta^0_{\alpha}$ relative to the diagrams of $\A$ and $\B$.  
For computable structures, the last two notions agree. 

A structure $\A$ is said to be \emph{ultrahomogeneous} if any isomorphism
between finitely generated substructures extends to an automorphism
of $\A$. Ultrahomogeneous structures were first studied
by Fra\"{i}ss\'{e} \cite{Fr86}, who defined the \emph{age} of a structure
to be the family of finitely generated substructures of $\A$ and
gave properties which characterized the age of an ultrahomogeneous
structure.  We define here a new notion of \emph{weakly homogeneous} structures,
where $\A$ is weakly ultrahomogeneous if there is a
finite (\emph{exceptional}) set  of elements $a_1,\dots,a_n$ of $\A$ such that
any isomorphism between finitely generated substructures  
$\A$ becomes ultrahomogeneous when constants representing these
elements are added to the language.  Csima, Harizanov, R. Miller and A. Montalban \cite{CHM11} studied
computable ages and the computability of the canonical
ultrahomogeneous structures, called \emph{Fra\"{i}ss\'{e} limits}.

The notion of a \emph{computably homogeneous} structure is defined in \cite{CHM11}, as follows.
Given a structure $\A$ and a tuple $\ov{a} = (a_1,\dots,a_n)$ from $\A$, let $\A_{\ov{a}}$ denote the substructure
of $\A$ generated by $\ov{a}$. $\A$ is said to be \emph{computably homogeneous} if there exists
a computable process which, given a tuple $\ov{a}$, a map $g: \ov{a} \to \A$, and $x \in \A$, returns a tuple
$\ov{b}$ and $y \in \A$ such that $\A_{\ov{a}} \subseteq \A_{\ov{b}}$ and, if $g$ extends to an 
embedding of $\A_{\ov{a}}$ into $\A$, then the function $h$, defined so that $h(x) = y$ and $h(g(a_i)) = a_i$ 
for each $i$, extends to an embedding of $\A_{x,g(\ov{a})}$ into $\A_{\ov{b}}$. We will say
that $\A$ is \emph{weakly computably homogeneous} if there is a finite tuple of elements $\ov{a}$ of $\A$ such
that $(\A,\ov{a})$ is computably homogeneous.  It is shown in \cite{CHM11} that if a computable structure
$\A$ is computably homogeneous, then every isomorphism between finitely generated substructures
extends to a computable automorphism of $\A$. If $\A$ is locally finite and ultrahomogeneous, then
$\A$ is in fact computably homogeneous.   

Here are some simple examples of countable ultrahomogeneous
structures.  See \cite{M02} for more details.  The linear ordering
$(\mathbb{Q},<)$ of the rationals is the unique ultrahomogeneous
countable linear ordering. The age here is just the set of all finite
linear orderings. This structure is computably categorical.  An
equivalence structure $(A,E)$ is ultrahomogeneous if and only if all
equivalence classes have the same size $k$, $1 \leq k \leq \aleph_0$;
then the age is the set of all finite equivalence structures with all
classes of size $\leq k$. These structures are computably categorical.

An injection structure is a set with a single 1-1 unary function.
This function induces a partition of the set into distinct orbits:
finite cycles, one-way infinite orbits ($\omega$-orbits), or two-way
infinite orbits ($\mathbb{Z}$-orbits).  An injection structure is
ultrahomogeneous if and only if it has no $\omega$-orbits.  For
example, there is the injection structure with infinitely many
${\mathbb Z}$-orbits, where the age is the set of structures
consisting of finitely many ${\mathbb Z}$-orbits. There is also the
injection structure with exactly one orbit of size $k$ for each finite
$k$, where the age is the family of finite injection structures with
no more than one orbit of any size $k$. The injection structure with infinitely many $\omega$-orbits is in fact not computably categorical, but is $\Delta^0_2$ categorical.

We observe that in the first two examples there are computable models
of the countable ultrahomogeneous structure.  In the third example,
one can have an arbitrary number of orbits of various finite sizes and
thus a structure which is not computable.

In this paper, we will closely examine the effective categoricity of 
ultrahomogeneous structures. In section 2, we will prove that any
computable ultrahomogeneous structure is relatively $\Delta^0_2$ categorical. 
We also introduce the notion of \emph{weakly ultrahomogeneous
  structures}, where $\A$ is weakly ultrahomogeneous if there is a
finite (\emph{exceptional}) set of elements $a_1,\dots,a_n$ such that
$\A$ becomes ultrahomogeneous when constants representing these
elements are added to the language. We show that any 
computable weakly ultrahomogeneous structure is relatively $\Delta^0_2$ categorical.

In section 3, we characterize the weakly ultrahomogeneous computable linear orders 
as those which have finitely many successivities, which is equivalent to being 
computably categorical. We also show that \emph{any} countable weakly ultrahomogeneous 
linear order has a computable copy. The notion of a \emph{minimal} exceptional set is introduced, 
and we characterize the minimal exceptional sets for weakly ultrahomogeneous linear orders.

In section 4, we characterize the countable weakly ultrahomogeneous equivalence structures as those in which all but finitely many equivalence classes have the same size. Here again every such structure has a computable copy, and 
a computable equivalence structure is weakly ultrahomogeneous if and only if it is computably categorical. 
The minimal exceptional sets for weakly ultrahomogeneous structures contain exactly one element from each of the exceptional sized equivalence classes.

In section 5, we characterize the weakly ultrahomogeneous injection structures as those having only finitely many $\omega$-orbits.  
The minimal exceptional sets contain exactly one member from each $\omega$-orbit. For injection structures, computable categoricity 
implies weak ultrahomogeneity, which implies $\Delta^0_2$ categoricity, but neither implication can be reversed.

In section 6, we consider weakly ultrahomogeneous graphs. It is shown that any countable weakly ultrahomogeneous graph is a disjoint union of
graphs $\HH$ and $\K$, where $\HH$ has finitely many components, and $\K$ consists of a finite or infinite number of complete graphs $K_n$ for a fixed $n \leq \aleph_0$. 
If every vertex of $\G$ has finite degree, then $\G$ is weakly ultrahomogeneous if and only if there is some fixed $n$ such that all but finitely many components of $\G$ are $K_n$. 
Since graphs are relational structures, every weakly homogeneous graph is computably categorical. However, there are computably categorical graphs which are not weakly ultrahomogeneous.

In section 7, we consider weakly homogeneous trees, among the family of countable trees of height $\leq \omega$. There are many formulations for the study of trees. A tree may be defined by a partial ordering,  a binary infimum function, a predecessor function, or a fixed number of successor functions.  
R. Miller \cite{Mil05} showed that no computable tree of infinite height can be computably categorical as a partial ordering (or in the infimum
framework).  Lempp, McCoy, Miller and Solomon \cite{LMMS} characterized the computably categorical trees of finite height, in the partial ordering 
setting.  Calvert, Knight and J. Millar \cite{CKM} defined a notion of \emph{rank homogeneity} for trees in the predecessor formulation, where trees of infinite height can be computably categorical.  In the partial ordering formulation, a tree $(T,\prec)$  is ultrahomogeneous if and only if it has rank $\leq 1$  (or equivalently height $\leq 1)$. $(T,\prec)$ is weakly ultrahomogeneous if and only if the set of elements of rank $\geq 1$ is finite. We give a characterization of the exceptional sets for a weakly ultrahomogeneous tree $(T,\prec)$.  A tree $(T,f)$ equipped with a predecessor function $f$ is ultrahomogeneous if and only if any two elements of the same height have an equal number of successors. We characterize the weakly ultrahomogeneous trees $(T,f)$ in general and illustrate this characterization for trees of height $\leq 3$.  

 In section 8, we consider $n$-equivalence structures, where a set comes with several equivalence relations, and in particular nested equivalence structures, where those relations are ordered by inclusion.  We show that if such a structure is ultrahomogeneous, then each individual equivalence relation is ultrahomogeneous. The converse does not hold in general, but we prove it for nested equivalence structures with all finite equivalence classes.  Nested $n$- equivalence structures turn out to be closely related to trees of height $n$, as studied by Leah Marshall \cite{Marshall}.  Given nested equivalence relations $E_1 \supset E_2 \dots \supset E_n$ on a set $A$, let $E_0 = A \times A$,  let $E_{n+1}$ be equality, and define the tree $T_{\A}$ to be the set of equivalence classes of $A$ under each $E_i$, ordered by reverse inclusion, so that $A$ is the root of the tree. Marshall shows that $\A$ is computably categorical if and only if $T_{\A}$ is computably categorical as a partial order.  We show that $\A$ is ultrahomogeneous if and only if $T_{\A}$ is ultrahomogeneous under predecessor and similarly for weak ultrahomogeneity. 

\emph{Index sets}  are an important tool for finding the complexity of a notion.
We define index sets for many of the structures studied here and determine the complexity
of the index sets corresponding to ultrahomogeneous and to weakly ultrahomogeneous structures. 
For these results, we use the standard enumeration from Soare \cite{Soare87} of the partial computable functions as $\{\varphi_e: e \in \omega\}$
and we let the $e^{th}$ computably enumerable set  $W_e$ equal the domain of $\phi_e$.  The set $W_e$ is defined as the union of stages $W_{e,s}$ where we assume that, for any $s$, there is at most one $e$ and one $n$ such that $n \in W_{e,s+1} - W_{e,s}$.  A set $B$ of natural numbers is said to be \emph{$\Pi^0_2$ complete}
if $B$ is itself a $\Pi^0_2$ set and if, for any $\Pi^0_2$ set $A$, there is a computable function $f$ such that $i \in A \iff f(i) \in B$ for all $i$; similar definitions apply for other definability classes, such as $\Sigma^0_3$.  Some well-known index sets are the $\Pi^0_2$ complete set $INF = \{e: W_e \ \text{is infinite}\}$ and the $\Sigma^0_3$ complete set 
$COF  = \{e: W_e \ \text{is cofinite}\}$. We show that the index set associated with weakly ultrahomogeneous structures of several types are $\Sigma^0_3$ complete; this includes linear orders, equivalence structures, and trees under predecessor. 

A preliminary version of this paper appeared in \cite{AC14}.  New material in the present paper includes the following.   In section 2, we extend the main theorem for weakly ultrahomogeneous structures which do not have computable copies.  In section 6, we examine ultrahomogeneous injection structures without computable copies.  There are two sections of new material,  section 6 on weakly ultrahomogeneous graphs and section 7 on weakly ultrahomogeneous trees. Section 8 includes new results on weakly ultrahomogeneous $n$-equivalence structures. The results on index sets for (weakly)  ultrahomogeneous structures are also new to this paper. 

The authors would like to thank the referee for very helpful comments.


\section{Categoricity of Ultrahomogeneous Structures}

In this section, we will show that any computable ultrahomogeneous
structure is $\Delta^0_2$ categorical.
Some lemmas are needed.  If $\mathcal{A}$ is an
ultrahomogeneous structure and $\mathcal{A}\cong\mathcal{B}$, by
composing maps it is easy to see that $\mathcal{B}$ is also ultrahomogeneous.
We also have the following stronger fact:

\begin{lemma}\label{iso} Let $\mathcal{A},\mathcal{B}$ be isomorphic
  ultrahomogeneous structures.  Let $X, Y$ be finitely generated
  substructures of $\mathcal{A},\mathcal{B}$ respectively.  If
  $\varphi:X\to Y$ is an isomorphism, then there is an isomorphism of
  $\mathcal{A}$ and $\mathcal{B}$ extending $\varphi$. \end{lemma}
\begin{proof} Let $\theta:\mathcal{A}\to\mathcal{B}$ be an
  isomorphism.  Then $\theta^{-1}\circ\varphi:X\to\mathcal{A}$ is an
  isomorphism of finitely generated substructures of $\mathcal{A}$.
  So it extends to an automorphism $\alpha:\mathcal{A}\to\mathcal{A}$.
  Then $\theta\circ\alpha:\mathcal{A}\to\mathcal{B}$ is an isomorphism
  that extends $\varphi$.
\end{proof}

Let $\mathcal{A}_s[x_1,\ldots,x_n]$ be the terms of height $s$
starting with the $x_i$, i.e. the set obtained by starting with the
elements of $\{x_1,\ldots,x_n\}$ and applying the functions of the
structure up to $s$-many times.  Then
$\la\vec{x}\ra=\bigcup_{s\in\omega}\mathcal{A}_s[\vec{x}]$.  While the
$\mathcal{A}_s[\vec{x}]$ aren't structures, we will say that
$\mathcal{A}_s[\vec{x}]\cong\mathcal{A}_s[\vec{y}]$ if for any terms
$t_1,\ldots t_m$ of height $s$ and any relation $R$, we have
$R(t_1[\vec{x}],\ldots,t_m[\vec{x}]) \Leftrightarrow
R(t_1[\vec{y}],\ldots,t_m[\vec{y}])$.

\begin{lemma} \label{lem2} $\la\vec{x}\ra\cong\la\vec{y}\ra$ with $x_i\to y_i$ iff
  for all $s\in\omega$ we have
  $\mathcal{A}_s[\vec{x}]\cong\mathcal{A}_s[\vec{y}]$. Thus,
  determining if two finitely generated substructures are isomorphic is
  $\pz$.\end{lemma}

\begin{proof}  For the left to right direction, it is clear that the
  restriction of the isomorphism to any height is an instance of the
  desired map.  The reverse implication follows from the fact that
  given finitely many terms $t_1[\vec{x}],\ldots t_m[\vec{x}]$, they
  occur by some finite height $s$. So the terms are in
  $\A_s[\vec{x}]$, hence for any relation $R$ we have
  $R(t_1[\vec{x}],\ldots,t_m[\vec{x}]) \Leftrightarrow
  R(t_1[\vec{y}],\ldots,t_m[\vec{y}])$ and so the map $t[\vec{x}]\to
  t[\vec{y}]$ is an isomorphism.
\end{proof}

\begin{theorem} \label{main} Every ultrahomogeneous structure is relatively
  $\Delta^0_2$ categorical. 
\end{theorem}
\begin{proof} Let $\mathcal{A}$ and  $\mathcal{B}$ be ultrahomogeneous structures and let $\varphi$ be an isomorphism from $\A$ to $\B$.
  We want to build an  isomorphism   $\theta$ which is $\Delta^0_2$ relative to the diagrams of $\A$ and $\B$.  We do this with a back-and-forth argument, building
  increasing partial isomorphisms $\theta_n$ at each stage and letting
  $\theta=\bigcup\theta_n$.  Let $a_0\in A$.  Since $\la
  a_0\ra\cong\la \varphi(a_0)\ra$, set
  $\theta_0(a_0)=\varphi(a_0)=b_0$.

\indent Suppose we have defined $\theta_{2n-1}$ for $\{a_0,\ldots
a_{2n-1}\}$ with $\theta_{2n-1}(a_i)=b_i$.  Choose the least
$a_{2n}\in A\setminus\{a_0,\ldots a_{2n-1}\}$.  There exists a $b\in
B$ such that $\la a_0,\ldots a_{2n}\ra\cong\la b_0,\ldots,b_{2n-1},
b\ra$ and we can choose the isomorphism so it extends $\theta_{2n-1}$
by Lemma \ref{iso}.  Now search for this $b$ using a $\pz$-oracle to check
whether $\la a_0,\ldots a_{2n}\ra\cong\la b_0,\ldots,b_{2n-1}, b\ra$, call
it $b_{2n}$, and define $\theta_{2n}(a_{2n})=b_{2n}$.

\indent If we have defined $\theta_{2n}$ for $\{a_0,\ldots
a_{2n}\}$ with $\theta_{2n}(a_i)=b_i$, we can similarly use a $\pz$-oracle to find an $a$ such that $\la a_0,\ldots
a_{2n}, a\ra\cong\la b_0,\ldots,b_{2n}\ra$. Call this $a_{2n+1}$, and
define $\theta_{2n+1}(a_{2n+1})=b_{2n+1}$.

\indent After constructing $\theta:\mathcal{A}\to\mathcal{B}$, we can
see that $\theta$ is a bijection since we took the least $a_i$ and
$b_i$ at each stage.  It is also clear that it is $\Delta^0_2$ relative to $\A$ and $\B$, since
$\theta$ is defined using an oracle which is $\pz$ in $\A$ and $\B$.  To show that it is an
isomorphism, fix an $m$-tuple $\vec{x}\subseteq \mathcal{A}$ such that $\theta(x_i)=y_i$
for each $x_i$ in $\vec{x}$, and let $\vec{y} =
\theta(\vec{x})$. Choose any relation $R$, and any function $f$ of
arity $m$.  Then for some $n$, $\vec{x}\subseteq\{a_0,\ldots, a_n\}$,
so $\vec{y}\subseteq\{b_0,\ldots,b_n\}$, and since $\la a_0,\ldots,
a_n\ra\cong\la b_0,\ldots,b_n\ra$, we have $R(\vec{x})\Leftrightarrow
R(\vec{y})$ and $\theta(f(\vec{x}))=f(\vec{y})$.
\end{proof}

The complexity of the isomorphism constructed in the theorem is a
direct result of the complexity of the problem of determining if two
finitely generated substructures are isomorphic.  So if this problem
is computable for a structure, then that structure is relatively computably
categorical.  If a structure is relational, then for any
finite subset $X$ of the universe we have $\la X\ra=X$ and checking if
two finite structures are isomorphic is computable. Therefore, all
relational ultrahomogeneous structures are relatively computably categorical. 

More generally, any ultrahomogeneous locally finite structure is relatively
computably categorical, where locally finite means that every finitely
generated substructure is finite. This is Proposition 4.1(3) of \cite{CHM11}. 
 The converse is false: an injection
structure consisting of a single $\Z$-orbit is computably categorical,
but is clearly not locally finite.

Next we introduce a weak version of ultrahomogeneity.
Whereas ultrahomogeneous structures have the property that `all points
look the same' in a very strong way, the weaker version will
allow finitely many elements to look different from the others.

\begin{defn} A structure $\mathcal{A}$ is weakly ultrahomogeneous if
  there exists a finite set $\{a_1,a_2,\ldots,a_n\}\subseteq A$ such
  that for all tuples $\vec{x},\vec{y}$ from $A$ with $\la
  \vec{a},\vec{x}\ra\cong\la\vec{a},\vec{y}\ra$ where each $a_i$ is
  fixed, this isomorphism of substructures extends to an automorphism
  of $\mathcal{A}$.  Call such a set $\{a_1,a_2,\ldots,a_n\}$ an
  exceptional set of $\mathcal{A}$.  
\end{defn}

Alternatively, $\A$ is weakly ultrahomogeneous if there is a finite
set $a_1,\dots,a_n$ of elements of $\A$ such that $(\A,a_1,\dots,a_n)$
is ultrahomogeneous in the extended language with constants for
$a_1,\dots,a_n$. Thus we can prove  the following.

\begin{theorem} \label{wmain} \begin{enumerate}
\item  Every weakly ultrahomogeneous structure is
relatively  $\Delta^0_2$ categorical. 
\item  Every locally finite weakly ultrahomogeneous structure is
relatively  computably categorical. In particular, every weakly ultrahomogeneous relational structure is
relatively computably categorical. 
\end{enumerate}
\end{theorem}

It is also easy to see that any locally finite weakly ultrahomogeneous structure is weakly
computably homogeneous. 

If $\A$ is a finite structure, it is trivially weakly ultrahomogeneous
since the universe can be taken to be an exceptional set.  Given any exceptional set, we can add finitely many
elements to it and obtain another exceptional set.  But more
interesting are the minimal exceptional sets and the senses in which
such sets are unique.  To see some instances of this definition, we
will look at the weakly ultrahomogeneous analogues of the examples of
ultrahomogeneous structures considered in the introduction.

In the remaining sections, we will examine specific families of structures. The goals are
to characterize ultrahomogeneous, and weakly ultrahomogeneous structures, and also to compare and contrast
these notions with effective categoricity. 

\section{Linear Orders}

We start with a characterization of computably categorical linear
orders proved by Remmel in \cite{R81}.  For a linear ordering $(A,<)$ and elements $a,b \in A$, $b$ is said to be the successor of $a$,
and $a$ the predecessor of $b$ if $a < b$ and there is no element $x$ with $a < x < b$. When $b$ is the successor of $a$, the pair $(a,b)$ is said to be a \emph{successor pair}, and each of  $a$ and $b$ is said to be a \emph{successivity}.
An element $a \in A$ is said to be a left endpoint if $a \leq  x$ for all $x \in A$ and is said to be a right endpoint if $x \leq a$ for all $x \in A$. These endpoints are unique if they exist.

\begin{theorem}[Remmel] For countable linear orders $\A$, the following are equivalent
\begin{enumerate}
\item  $\A$ is  computably  categorical.
\item  $\A$ is  relatively computably
  categorical.
\item $\A$ has finitely many successivities.
\item $\mathcal{A}=L_0+\Q+L_1+\Q+\ldots+\Q+L_n$ where the $L_i$ are
  finite chains, $L_0,L_n$ 
are possibly empty and $|L_i|\geq2$ for $1\leq i\leq n-1$.
\end{enumerate} \end{theorem}

Next we give a characterization of weakly ultrahomogeneous linear orders.

\begin{theorem} \label{wulo} A countable linear order $\mathcal{A}$ is weakly ultrahomogeneous
iff $\mathcal{A}$ has finitely many successivities.
\end{theorem}

\begin{proof}
First, suppose $\mathcal{A}$ has infinitely many
  successivities.  Let $\{a_1,\ldots, a_n\}$ be a finite subset of
  $A$; we will show it cannot be an exceptional set.  The
  successivities of $\mathcal{A}$ occur in finite chains or in subsets
  of order type $\omega, \omega^*$ (the reverse order of $\omega$), or
  $\Z$.  If there is a set $C$ of successivities that has order type $\omega$ or
  $\Z$, choose elements $x_1<x_2$ and $y_1< y_2\in C$ greater than all
  $a_i\in C$ such that there are more elements between $x_1$ and $x_2$
  than there are between $y_1$ and $y_2$.  Then
  $\la\vec{a},x_1,x_2\ra\cong\la\vec{a},y_1,y_2\ra$, but the
  isomorphism can't be extended to an automorphism.  If $C$ is of
  order type $\omega^*$, we can repeat the argument above by choosing
  the elements below all $a_i$ in $C$.  Finally, if $\mathcal{A}$ has
  infinitely many successivities in finite chains, choose one of these
  chains containing none of the $a_i$ and from it choose the first
  element $x$ and the second element $y$.  Then
  $\la\vec{a},x\ra\cong\la\vec{a},y\ra$, but the isomorphism can't be
  extended to an automorphism.  With infinitely many successivities
  one of these situations must occur, and in any case we see
  $\{a_1,\ldots, a_n\}$ isn't exceptional, so $\mathcal{A}$ isn't
  weakly ultrahomogeneous.  \\ \\
For the other direction, let $S=\{a_1,\ldots,a_k\}$ be the set of successivities and endpoints of $\A$. We claim that this set is exceptional.
  To see this, suppose  $\la\vec{a},\vec{x}\ra\cong\la\vec{a},\vec{y}\ra$ 
with $a_i\to a_i$  for $i\leq k$ and $x_j\to y_j$ for $j\leq n$.  
Also assume that  $x_j,y_j\notin S$ for $j\leq n$.  So for each $j\leq
n$, $x_j$ and $y_j$
  are in the same copy of $\Q$, since they bear the same relation to
  all elements of $S$.  Then use the ultrahomogeneity within each copy of $\Q$ containing some $x_i, y_i$ and the identity elsewhere to get an automorphism of the whole structure.
\end{proof}

\begin{corollary} A computable linear order $\mathcal{A}$ is
  computably categorical iff $\mathcal{A}$ is weakly
  ultrahomogeneous. \end{corollary}

\begin{corollary} For any countable weakly homogeneous linear order $\A$, 
there is a computable structure isomorphic to $\A$.
\end{corollary}

Now we consider index sets for linear orderings.  The $e$th computable linear order $\A_e = (\omega, <_e)$ is given by the
$e$th partial recursive function $\phi_e$ when $\phi_e$ is total and is 
the characteristic function of a linear ordering. It is easy to see that $LIN = \{e: \A_e \ \text{is a linear ordering}\}$ 
is a $\Pi^0_2$ set.  

Let $UHL = \{e: \A_e\ \text{is an ultrahomogeneous linear ordering}\}$ and let 
\\$WUL = \{e: \A_e\ \text{is a weakly ultrahomogeneous linear ordering}\}$.

\begin{theorem}
\begin{enumerate}
\item[(a)] The index set $UHL$  is $\Pi^0_2$ complete,
and in fact $\Pi^0_2$ complete relative to $LIN$. 
\item[(b)] The index set $WUL$ is $\Sigma^0_3$ complete,
and in fact $\Sigma^0_3$ complete relative to $LIN$. 
\end{enumerate}
\end{theorem}

\begin{proof} (a) $UHL$ is a $\Pi^0_2$ set,  since $\A_e$ is ultrahomogeneous if and  only if it is isomorophic to $(\Q,<)$, that is, if and only if it is dense and without endpoints. 
 For the completeness, we give a reduction $f$ of the $\Pi^0_2$ complete set $INF$ to
$UHL$ in such a way that $f(e) \in LIN$ for every $e$. The construction of $\A_{f(e)}$ is in stages $\A^s$ as follows (we suppress the subscripts for ease
of comprehension).  $\A_0 = \{0\}$.  After any stage $s$, we will have a linear order $A^s = \{a^s_0 < a^s_1 < \dots < a^s_s\}$.  At any stage $s+1 = 3t+1$,
we ensure that $\A$ will have no least element by letting $a^{s+1}_0 = s+1$ and $a^{s+1}_{i+1} = a^s_i$ for all $i \leq s$.  Similarly  at any stage $s+1  = 3t+2$,
we ensure that $\A$ will have no greatest element by letting $a^{s+1}_{s+1} = s+1$ and $a^{s+1}_i = a^s_i$ for all $i \leq s$. Finally, at any stage $s+1 = 3t+3$,
we ensure that $\A$ will be densely ordered, if $W_e$ is infinite, as follows.  There are two cases.  If an element enters $W_e$ at stage $t+1$, do the following:  Let $\langle a^s_j, a^s_{j+1} \rangle$ be the least code for a successor pair in $\A^s$ and put $s+1$ between them, so that $a^{s+1}_i = a^s_i$ for $i \leq j$, $a^{s+1}_{j+1} = s+1$, and $a^{s+1}_{i+1} = a^s_i$ for $i > j$. If no element enters $W_e$ at stage $t+1$, then make $s+1$ smaller than all elements from $\A_s$ as in the case for $s+1 = 3t+1$ above. 

If $W_e$ is infinite, then it follows from the construction that $\A_{f(e)}$ will be a dense linear ordering without endpoints. If $W_e$ is finite, then after some stage $s$,
no new elements enter $W_e$, so that the block of successors from $a_0^s$ to $a_s^s$ is preserved and $\A_{f(e)}$ will be isomorphic to $\Z$. 

(b) It is easy to see, by quantifying over the finite exceptional set, that $WUL$ is a $\Sigma^0_3$ set. For the completeness, we give a reduction $g$ of the $\Sigma^0_3$ complete set $COF$ to $WUL$ in such a way that $g(e) \in LIN$ for every $e$. The construction of $\A_{g(e)}$ is a modification of the construction in part (a). The difference is that in the case $s +1 = 3t+3$, we look for the least code $\langle a^s_j, a^s_{j+1} \rangle$ for  a successor pair  such that both elements have come into $W_e$ by stage $t+1$. 

It follows that for any pair $a,b$ such that $b$ is the successor of $a$ in $\A_{g(e)}$, either $a$ or $b$ is not in $W_e$.  Thus, if  $W_e$ is cofinite, then $\A_{g(e)}$ has 
only finitely many successivities. On the other hand, if $a \notin W_e$, then it has either a successor or a predecessor when it comes into $\A_e$, and by the new requirement no element is ever put in between.  Thus, if $W_e$ is not cofinite, then $\A_{g(e)}$ will have infinitely many successivities. 
\end{proof}

Let us say that an exceptional set $S$ for a weakly ultrahomogeneous
structure is a \emph{minimal exceptional set} if no proper subset of
$S$ is exceptional.  Such as set must exist since exceptional sets are finite. We
will try to characterize minimal exceptional sets and determine
whether they are unique to a structure, or perhaps unique up to
automorphism.  Let us say that $a$ is a \emph{special point} of a linear order $\A$ is it is either an endpoint or a successivity.
Thus the set of special points of a countable weakly ultrahomogeneous linear order $\mathcal{A}=L_0+\Q+L_1+\Q+\ldots+\Q+L_n$ is $L_0 \cup L_1 \cup \dots \cup L_n$. 
 
For a weakly ultrahomogeneous linear order, the set
of all special points is an exceptional set, but this set is not
necessarily minimal.  A characterization of exceptional sets of special points  is as follows.

\begin{proposition} \label{prop1}  Let $\mathcal{A}=L_0+\Q+L_1+\Q+\ldots+\Q+L_n$ be a
  countable weakly ultrahomogeneous linear order.  A set of $S=\{a_1,\ldots,a_n\}$ of special points is exceptional iff it satisfies the following conditions:
  \begin{enumerate}
\item[(i)]$A \setminus S$ does not include any successor pair.
\item[(ii)] $S$ contains each last element of $L_0,\ldots L_{n-1}$ and 
each first element of $L_1,\ldots L_n$.
\end{enumerate}
\end{proposition}

\begin{proof}  If (i) fails for $S$, let $x,y \notin S$, with $y$ the successor of $x$, so that $x,y$ belong to
some $L_i$.  Then $\la \vec{a},x \ra \cong\la\vec{a},y \ra$,
 but the isomorphism can't extend since $x$ and $y$ are in different
 positions in the finite sequence $L_i$.  

If $(ii)$ fails for $S$,  let $x$ witness its failure and let $y$ be an element of the copy of
  $\Q$ adjacent to $x$.  Then $\la \vec{a},x\ra\cong\la\vec{a},y\ra$,
  but the isomorphism can't extend since $x$ is a successivity and $y$
  isn't.  

Now suppose $i)$ and $ii)$ hold for
  $S$ and suppose   $\la\vec{a},\vec{x}\ra\cong\la\vec{a},\vec{y}\ra$ by $\varphi$, where each $x_i$ and $y_j$ is not in $S$.
 Any   $x_j$ which is a special point is either uniquely between two $a_i$, or
  is the unique element greater than or less than all $a_i$ if it is
  an endpoint. Hence $y_j = \varphi(x_j) = x_j$. Any non-special points
  $x_j,y_j$ must be in the same copy of $\Q$ and it follows from the
  ultrahomogeneity of $\Q$ that we can extend $\varphi$ to an
  automorphism.
\end{proof}

Given an exceptional set, a (possibly) smaller exceptional set is obtained by removing all non-special points. So the above proposition says that minimal exceptional sets are sets of special points  satisfying the two conditions while no proper subset satisfies them both.

As an example, consider the linear order $\Q+L_0+\Q$ where $L_0=\{a_1<
a_2<\ldots< a_5\}$.  Then both $\{a_1,a_3,a_5\}$ and
$\{a_1,a_2,a_4,a_5\}$ are minimal exceptional sets.  This shows that,
while we would like the exceptional sets of a weakly ultrahomogeneous
structure to be unique in some way, minimal exceptional sets aren't
necessarily unique and in fact need not even be isomorphic.
\\ \\\indent Recall that in a structure $\A$, an element $b\in A$ is
definable from a set $S\subseteq A$ if $\{b\}$ is a subset of $A$
definable from $S$.  Let the definable closure of $S$ be $D(S)=\{x\in
A: x \mbox{ is definable from }S\}$.  An important fact about
definability we will use repeatedly is that if $b$ is definable from
$S$ and $\sigma$ is an automorphism of $\A$ fixing all elements of
$S$, then $\sigma$ also fixes $b$.  Looking at definable closures
reveals a sense in which minimal exceptional sets are unique.

\begin{proposition} \label{prop2} Let $\mathcal{A}$ be a weakly ultrahomogeneous
  linear order and let $M=\{a_1<\ldots<a_n\}$ be a minimal exceptional
  set.  Then $D(M)$ is the set of special points of
  $\mathcal{A}$. \end{proposition}

\begin{proof}
Suppose $x$ is in one of the copies of $\Q$ in $\mathcal{A}$.  Using
the ultrahomogeneity of $\Q$, there is an automorphism of
$\mathcal{A}$ moving $x$ while fixing $M$.  So $x\notin D(M)$.  Now
let $x$ be a special point of $\mathcal{A}$; we may assume $x\notin M$.
Then $x$ is the unique element satisfying $\phi(y) : (y<a_i)$ if it is a
left endpoint of $\mathcal{A}$, similarly for right endpoints, or else
it is the unique element satisfying $\phi(y) : (a_i<y<a_{i+1})$ for some
$i\leq n$.
\end{proof}

\section{Equivalence Structures}

The effective categoricity of equivalence structures was investigated
by Calvert, Cenzer, Harizanov and Morozov in \cite{CCHM06}. The
\emph{character} of an equivalence structure indicates the number of
equivalence classes of each size. The structure is said to have
\emph{bounded character} if there exists a $k\in\mathbb{N}$ such that 
all finite classes have size at most $k$. It is proved in
\cite{CCHM06} that an equivalence structure $\A$
is computably categorical iff $\A$ has finitely many finite classes,
or $\A$ has finitely many infinite classes, bounded character, and
there is at most one $k$ such that there are infinitely many classes
of size $k$.  This condition is equivalent to saying that all but finitely many
classes of $\A$ have the same size.  For computable equivalence structures,
computable categoricity is the same as relative computable categoricity. 

\begin{theorem}
A countable equivalence structure $\A$ is weakly ultrahomogeneous iff all but finitely many equivalence classes of $\A$ have the same size.
In this case, a minimal exceptional set contains exactly one element
from each of the exceptional equivalence classes.
\end{theorem}

\begin{proof}
For one direction, suppose that $\A$ has infinitely many classes of
  different sizes and let $\{a_1,\ldots,a_n\}$ be a finite subset.
  Then find elements $x,y$ from classes of different sizes so neither
  is related to any $a_i$.  Then
  $\la\vec{a},x\ra\cong\la\vec{a},y\ra$, but this can't extend to an
  automorphism.  \\ \\ Now for the reverse implication. If all but finitely many
  equivalence classes of $\A$ are of the same size, let
  $\{a_1,\ldots,a_n\}$ contain exactly one element from each of these
  exceptional classes.  Then suppose
  $\la\vec{a},\vec{x}\ra\cong\la\vec{a},\vec{y}\ra$ via the
  isomorphism $\varphi$ so $\varphi(x_i)= y_i$ and $\varphi(a_k)=a_k$.
  Either each pair $x_i,y_i$ is in an equivalence class with some
  $a_k$, or if not they are in classes of the same size.  Then the
  isomorphism extends to an automorphism as follows. 

First, we will explain how the equivalence classes are mapped, and
then we will describe what happens to the elements of each class. 
Fix each exceptional class as well as any classes which do not
contain any $x_i$ or $y_i$.  If a nonexceptional class has an $x_i$,
then map $x_i$ to $y_i$ and hence the class $[x_i]$ to the
class $[y_i]$. If there are nonexceptional classes with a $y_i$ but no
$x_i$, there must be the same number of nonexceptional classes with an
$x_i$ but no $y_i$.  Send each class of the first kind to one of the
second kind. Thus we may end up with cycles, say $[x_1]$ maps to $[y_1]$
and $y_1 E x_3$, so that $[y_1]$ maps to $[y_3]$, and then $[y_3]$ maps to $[x_1]$.

Within each class, do the following. For classes with no $x_i$ or
$y_i$, the class is fixed and we also fix each element of the
class. For the nonexceptional classes containing some of the $x_i$ or
$y_i$, we have mapped the $x_i$ and $y_i$ above, and the remaining
elements can be mapped arbitrarily. For the exceptional classes
containing some of the $x_i$ or $y_i$, it follows that $\varphi(x_i) E
x_i$, thus we can map those elements respecting $\varphi$ using a
cycle decomposition similar to that described above for the
nonexceptional classes. Now the remaining elements can simply be
fixed.

The claim about the minimal exceptional sets follows since the proof
shows such a set is exceptional, and that a finite set disjoint from
two classes of different sizes is not exceptional.
\end{proof}

\begin{corollary} A computable equivalence structure $\A$ is 
weakly ultrahomogeneous iff $\A$ is computably categorical. 
\end{corollary}

\begin{corollary} For any countable weakly homogeneous equivalence
  structure $\A$, 
there is a computable structure isomorphic to $\A$.
\end{corollary}

With this characterization of weakly ultrahomogeneous structures and
their minimal exceptional sets, we can again investigate their
uniqueness properties.  Given two minimal exceptional sets, there is
an automorphism of the structure sending one to the other by
interchanging the two elements in each exceptional class and fixing
everything else.  However, as opposed to linear orders we don't have
uniqueness of the definable closures.

\begin{proposition} \label{prop3} Let $\A=(A,E)$ be a weakly ultrahomogeneous 
equivalence structure and let $S=\{a_1,\ldots,a_n\}$ a minimal
exceptional set. Then $x\in A$ is definable from $S$ iff $x\in S$ or
$x$ is in an exceptional class of size at most 2. 
\end{proposition}

\begin{proof} If $x\in A$ isn't in an exceptional class, there is an 
automorphism fixing all the exceptional classes but moving $x$ by
interchanging $[x]$ with another class of the same size.  If $x$ is in
an exceptional class of size at most 2, then for some $i\leq n$,
either $x = a_i$ or $\{x\} = \{y: yEa_i\ \&\ y \neq a_i\}$.  If $x$ is
in an exceptional class of size greater than 2, then $[x]$ contains
$a_i$ for some $i\leq n$ and an element $y$ distinct from $x$ and
$a_i$.  In this case, switching $x$ and $y$ and using the identity
everywhere else is an automorphism of $\A$ moving $x$ and fixing $S$.
\end{proof}

Now we consider index sets for equivalence structures.  
The $e$th computable equivalence structure $\A_e = (\omega, \equiv_e)$ is given by the
$e$th partial recursive function $\phi_e$ when $\phi_e$ is total and is 
the characteristic function of an equivalence relation. It is easy to see that 
$EQ= \{e: \A_e \ \text{is an equivalence structure}\}$ 
is a $\Pi^0_2$ set.  Let $UHQ = \{e: \A_e\ \text{is an ultrahomogeneous equivalence structure}\}$ and let  
$WUQ= \{e: \A_e\ \text{is a weakly ultrahomogeneous equivalence structure}\}$.

\begin{theorem}
\begin{enumerate}
\item[(a)] The index set $UHQ$ is $\Pi^0_2$ complete,
and in fact $\Pi^0_2$ complete relative to $EQ$. 
\item[(b)] The index set $WUQ$ is $\Sigma^0_3$ complete,
and in fact $\Sigma^0_3$ complete relative to $EQ$. 
\end{enumerate}
\end{theorem}

\begin{proof} (a) It is easy to see that $UHQ$ is a $\Pi^0_2$ set. For the completeness, we give a reduction $f$ of the $\Pi^0_2$ complete set $INF$ to $UHQ$ in such a way that $f(e) \in EQ$ for every $e$. The idea of the construction is to make every equivalence class have size two if $W_e$ is infinite, and otherwise to have finitely many classes of size two and the rest of size one.  The construction of $\A_{f(e)}$ is in stages $\A^s$ as follows (we suppress the subscripts for ease of comprehension).  $\A^0 = \{0,1,2\}$ with $0 \equiv 1$. This ensures that there will always be at least one class of size two.  After any stage $s$, we will have an equivalence structure $\A^s$ with finite universe $A_s = \{0,1,\dots,2s+2\}$ with $n +1 \leq s+1$ classes of size two, where $n = card(W_{e,s})$, and at least one class of size one; in particular $2s+2$ will make up a class of size one. We assume that, for all $e$ and $s$, 
$W_{e,s} \subseteq \{0,1,\dots,s-1\}$ and that at most one element comes into $W_e$ at any stage $s+1$. There are two cases in the construction at stage $s+1$.  First, suppose that  a new element comes into $W_e$ at stage $s+1$, which is the $n+1$st element of $W_e$.  In this case, we look for the least $i \leq 2s$ not in a class of size two and make $2s+3 \equiv i$, so that $2s+4$ is in a class by itself.  Next, suppose that no new element comes into $W_e$ at stage $s+1$.  Then we simply let $A^{s+1} = \A_s \cup \{2s+3,2s+4\}$ without adding any pairs to the equivalence relation. 

If $W_e$ is infinite, then it follows from the construction that every equivalence class of $\A_{f(e)}$ will have size two, so that $\A_{f(e)}$ is ultrahomogeneous.  If $W_e$ is finite, then all but finitely many classes will have size one, but there will be at least one class of size two, so that $\A_{f(e)}$ is not ultrahomogeneous. 

(b) It is easy to see, by quantifying over the finite exceptional set, that $WUQ$ is a $\Sigma^0_3$ set. For the completeness, we give a reduction $g$ of the $\Sigma^0_3$ complete set $COF$ to $WUQ$ in such a way that $g(e) \in EQ$ for every $e$. 

The construction of $\A_{f(e)}$ is a modification of the construction in part (a). The difference is that we make $2n$ have a class of size $2$ if and only if $n \in W_e$ while each odd number has a class of size 2. At stage 0, we have $\A^0 = \emptyset$. After stage $s$, we have $\A^s = \{0,1,\dots,2s-1\}$ so that 
\begin{enumerate}
\item[(i)] for all $a< 2s$, $[a]^s$ has size either one or two;
\item[(ii)] for all $n < s$, $[2n]^s = \{2n\}$ if and only if $n \notin W_{e,s}$; 
\end{enumerate}
At stage $s+1$, there are two cases.
\smallskip

Case I: If there is an $n\leq s$ such that $n\in W_{e,s+1}\setminus W_{e,s}$, make $2n$ equivalent to the least odd number $2m+1\leq 2s+1$ not in a class of size two. If $m<s$, put $2s+1$ in a class of size 1 and if $n<s$, put $2s$ in a class of size one.
\smallskip

Case II: If $W_{e,s+1}=W_{e,s}$, put $2s$ in a class of size one. If there is an odd number less than $2s+1$ in a class of size one, pair that number with $2s+1$. If not, put $2s+1$ in a class of size one.

%
%
%
%

It follows from the construction that, for each $n$, $[2n]$ has size two if and only if $n \in W_e$ and that $[2n+1]$ has size two for every $n$. 
If $W_e$ is cofinite, then all but finitely many classes have size two, and hence $\A_{f(e)}$ is weakly ultrahomogeneous.  If $W_e$ is not cofinite, then $\A_{f(e)}$ has infinitely many classes of size one and infinitely many classes of size two, and hence is not weakly ultrahomogeneous.
\end{proof}

\section{Injection Structures} 

The effective categoricity of injection structures was studied by
Cenzer, Harizanov and Remmel in \cite{CHR14}. It was shown that an
injection structure is computably categorical if and only if it has
finitely many infinite orbits, and is $\Delta^0_2$ categorical if and
only if it either has only finitely many orbits of type $\Z$ or has
only finitely many orbits of type $\omega$. In each case, relative categoricity 
is the same as categoricity. 

\begin{proposition} An injection structure is ultrahomogeneous 
if and only if it has no $\omega$-orbits. 
\end{proposition}

\begin{proof}
Let $\A = (\omega,f)$ be an injection structure. Suppose first that $\A$ has an $\omega$-orbit and let $a$ be 
the element of this orbit with no predecessor.  Then the substructures $\A = \langle a\rangle = \{f^n(a): n \in \omega\}$ and $\langle f(a)\rangle = \{f^{n+1}(a): n \in \omega\}$
are isomorphic but this clearly cannot be extended to an automorphism of $\A$.  

For the other direction, suppose that $\A$ has no $\omega$-orbits and let $\phi$ be an isomorphism mapping $\langle a_1,\dots,a_n\rangle$ to $\langle b_1,\dots,b_n\rangle$. 
If $a_i$ belongs to an orbit of finite size $k$, then the orbit of $b_i$ also of size $k$ and $\phi$ maps $\langle a_i\rangle$ to $\langle b_i\rangle$.  If $a_i$ belongs to an orbit of type $\Z$, then so does $b_i$ and $\phi$ maps $f^j(a_i)$ to $f^j(b_i)$ for each $j$. This isomorphism can be extended to the entire orbit of $a_i$ by mapping $f^{-n}(a_i)$ to $f^{-n}(b_i)$.  Finally, this can be extended to an automorphism of $\A$ by mapping $a$ to $a$ for elements of any orbits not among the orbits of $a_1,\dots,a_n$. 
\end{proof}

Next we consider weakly ultrahomogeneous injection structures. 

\begin{proposition} \label{prop6} A countable injection structure $\A$ is weakly 
ultrahomogeneous iff it has finitely many $\omega$-orbits. In this
case, a minimal exceptional set contains exactly one member from each
$\omega$-orbit.
\end{proposition}

\begin{proof}
Suppose that $\A$ is an injection structure having only finitely many
$\omega$-orbits.  Let $\{a_1,\ldots,a_n\}$ contain exactly one element
from each of the $\omega$-orbits, and assume that
$\la\vec{a},\vec{x}\ra\cong\la\vec{a},\vec{y}\ra$ via the isomorphism
$\varphi$. The isomorphism is extended to an automorphism as follows.

First, orbits not containing any $x_i$ or $y_i$ are fixed.  If $x_i$
is in a finite orbit of size $k$, then $y_i$ is also in a finite orbit
of size $k$, and the orbit of $x_i$ is mapped to the orbit of
$y_i$. If there are finite orbits of size $k$ containing some $y_j$
but no $x_i$, then there must be an equal number of orbits of size $k$
containing some $x_i$ but no $y_j$, and then we map each class of the
first kind to one of the second kind.

If $x_i$ is in a $\Z$-orbit, then $\varphi$ maps the sequence
  $(x_i,f(x_i),\dots)$ to the sequence $(y_i,f(y_i),\dots)$ and this can be
  extended to the entire orbits. 
 Each $\omega$-orbit must be fixed, since it contains one of the
  $a_i$, and $\varphi$ fixes $a_i$ and respects $f$. 

Now assume $\A$ has infinitely
many $\omega$-orbits, and let $\{a_1,\ldots,a_n\}$ be a finite set.
In an $\omega$-orbit containing none of the $a_i$, let $x_0$ be the
initial element and $x_1=f(x_0)$.  Then
$\la\vec{a},x_0\ra\cong\la\vec{a},x_1\ra$, but the isomorphism can't
extend since $x_1$ is in the range of $f$ while $x_0$ isn't.  Thus
$\A$ isn't weakly ultrahomogeneous.  

If a finite set $S$ doesn't
include an element from each $\omega$-orbit, we may repeat the above
argument with the orbit not intersecting $S$ to show the finite set
isn't exceptional.  The claim about minimal exceptional sets follows.
\end{proof}

It follows that, for computable injection structures, computable
categoricity implies weak ultrahomogeneity which implies
$\dt$ categoricity.  Neither implication can be reversed as witnessed
by computable injection structures consisting of only infinitely many
$\Z$-orbits, and of only infinitely many $\omega$-orbits,
respectively.

In contrast to the results for linear orders and equivalence structures, there exist 
ultrahomogeneous injection structures which are not isomorphic to computable structures. 

The character $K(\A)$ for an injection structure $\A = (\omega,f)$ is defined by
\[
K(\A) = \{(n,k): \A\ \text{has at least $n$ orbits of size}\ k\}.
\]
It is easy to see that $K$ will be the character of an injection structure if and only if it is a subset of $\omega \times (\omega - \{0\})$ 
such that, for all natural numbers $n$ and $k$, if $(n+1,k) \in K$ then $(n,k) \in K$. 

\begin{proposition} \begin{enumerate}
\item For any character $K$, there is an ultrahomogeneous injection structure $\A$ 
with character $K$ and with an arbitrary finite number of  orbits of type $\Z$.  Furthermore, $\A$ is 
relatively computably categorical. 
\item For any character $K$, there is a weakly  ultrahomogeneous injection structure $\A$ 
with character $K$ and with an arbitrary finite number of  orbits of type $\Z$ and an arbitrary finite number
of orbits of type $\omega$.  Furthermore, $\A$ is 
relatively computably categorical. 
\item For any character $K$, there is an ultrahomogeneous injection structure $\B$ 
with character $K$ and with an infinite number of  orbits of type $\Z$. 
 Furthermore, $\A$ is 
relatively $\Delta^0_2$ categorical. 
\item For any character $K$, there is a weakly ultrahomogeneous injection structure $\B$ 
with character $K$ and with an arbitrary number of  orbits of type $\Z$ and an arbitrary finite number of orbits of type $\omega$. 
 Furthermore, $\A$ is 
relatively $\Delta^0_2$ categorical. 
\end{enumerate}
\end{proposition}

\begin{proof} In each case,  it is clear that such structures exist.  The effective categoricity follows from Theorem \ref{wmain} with an additional argument in the second case. That is, for each of the finitely many orbits of type $\omega$, we simply identify the initial elements of each orbit and use this to define the  mapping on the orbits of type $\omega$. 
\end{proof}

Again in contrast to linear orderings and equivalence structures, injection structures are not always locally finite.  
Thus it is possible for an ultrahomogeneous computable injection structure to fail to be computably homogeneous. 
Of course a structure with only finite orbits is locally finite, and therefore a computable injection structure with only finite orbits will be
computably homogeneous. A computable structure with finitely many infinite orbits will be weakly computably homogeneous. 

However, a structure with infinitely many $\Z$-orbits may or may not be computably homogeneous.  

\begin{proposition} There is a computably homogeneous injection structure consisting of infinitely many $\Z$-orbits. 
\end{proposition}

\begin{proof}  Consider the $\Z$-chain with universe $\{2i+1: i \in \omega\}$ and function $f$ defined so $f(4i+1) = 4i+5$, $f(4i+7) = 4i+3$, and $f(3) = 1$.
Now, for each $n$, simply multiply each element of this chain by $2^n$ to create infinitely many $\Z$-orbits. This structure is clearly computably homogeneous. 
\end{proof}

\begin{proposition} \label{proporb} There is a computable injection structure $\A$, consisting of infinitely many $\Z$-orbits, such that, for each $e$, 
there are infinitely many orbits which are Turing equivalent to $W_e$.  
\end{proposition}

\begin{proof}
We will build $\A = (A,f)$ so that the orbit of $2^e$ is a c.e. set which is Turing equivalent to $W_e$.  The function $f$ will be constructed in uniformly computable 
\emph{stages} $f_s$ for $s \in \omega$, so that $f = \bigcup_s f_s$.  The orbit of an element $a$ under $f$ will be denoted by $\Or(a)$ and the orbit of $a$ under $f_s$ will be denoted by $\Or_s(a)$.  $\A$ will have universe $\{2^e 3^i 5^j, 2^e 3^i 7^j:  e,i,j, \in \omega\}$. The orbit $\Or(2^e)$ will consist of $\{2^e 5^j, 2^e 7^j: j \in \omega\}$ together with $\{2^e 3^{i+1} 5^j,  2^e 3^{i+1} 7^j: j \in \omega,  i \in W_e\}$. Thus $i \in W_e$ if and only if $2^e 3^{i+1} \in  \Or(2^e)$, so that $W_e$ is one-one reducible to $\Or(2^e)$.  For the other direction, the orbit of $2^e$ as described above is clearly Turing reducible to $W_e$.  The construction of the mapping $f$ is as follows. For each $e$,  begin to build the function $f$ by letting  $f(2^e 5^{j+1}) = 2^e 5^j$ and $f(2^e 7^j) = 2^e 7^{j+1}$.  At the same time, build auxiliary orbits for each $2^e 3^{i+1}$ by letting  $f(2^e 3^{i+1} 5^{j+1}) = 2^e 3^{i+1} 5^j$ and $f(2^e 3^{i+1} 7^j) = 2^e 3^{i+1} 7^{j+1}$. When an element $i$ comes into $W_e$ at stage $s+1$, insert the partial orbit of $\Or_s(2^e 3^{i+1})$ onto the end of $\Or_s(2^e)$. After that, we resume building $\Or(2^e)$ by putting $2^e 7^{s+2}$ after the inserted part, while continuing to include in $\Or(2^e)$  elements of the form  $2^e 3^{i+1} 5^j$  and $2^e 3^{i+1} 7^j$ where $i \in W_{e,s+1}$ as the construction continues.

Details of the construction are as follows. For the sake of simplicity, we fix $e$ and describe the construction of the orbits of $2^e 3^i$, for all $i$. 
At stage $s$, we will have defined $f$ on $\{2^e 3^i 5^j, 2^e 3^i 7^j: i,j \leq s\}$.  Thus at stage $0$ we have only $f_0(2^e) = 2^e7$. After stage $s$, $\Or(2^e)$ will have initial element $2^e 5^s$ and final element $2^e 7^{s+1}$.  We will assume as usual that at any stage $s+1$, at most one element $x$ enters any c.e. set $W_e$ and that $x \leq s$.  At stage $s+1$, extend $f_s$ to $f_{s+1}$ as follows. For any $i \leq s$ such that $i \notin W_{e,s+1}$, simply map $2^e 3^{i+1}5^{s+1}$ to $2^e 3^{i+1} 5^s$ and map $ 2^e 3^{i+1} 7^{s+1}$ to  $2^e 3^{i+1} 7^{s+2}$.   For the orbit of $2^e$, first map
 $2^e 5^{s+1}$ to $2^e 5^s$. If some  $i \leq s$ comes into $W_e$ at stage $s+1$, then insert the current orbit of $\Or_s(2^e 3^{i+1})$ into  $\Or(2^e)$ 
right after $2^e 7^{s+1}$.  Next let $\{i_0,\dots,i_{k-1}\} = W_{e,s+1}$ and, if $W_{e,s+1}$ is not empty, extend $\Or(2^e)$ by putting the sequence  $2^e 3^{i_0+1} 5^{s+1}, 2^e3^{i_0+1} 7^{s+2}, 2^e 3^{i_1+1} 5^{s+1}, \dots, 2^e 3^{i_{k-1}+1} 7^{s+2}$ onto the end.  Finally, put $2^e 7^{s+2}$ at the very end of the orbit.  

It follows from the construction that, for each $e,i,j$, $f_{s+1}$ is defined on $\{2^e 3^i 5^j, 2^e 3^i 7^j: i,j \leq s\}$ by stage $s$, so that the map
$f =  \cup_s f_s$ is computable.  It is clear from the construction that each orbit is of type $\Z$ as described above, so that the orbit of $2^e$ is Turing equivalent to $W_e$.
\end{proof} 

\begin{theorem} There is a computable injection structure $\A$, consisting of infinitely many $\Z$-orbits, which is not weakly computably homogeneous. 
\end{theorem} 

\begin{proof}   Let $\A = (A,f)$ be the injection structure from Proposition \ref{proporb}.  Fixing any finite number of orbits, there are still
two orbits of different degree so that the isomorphism between these two orbits cannot be extended to a computable automorphism.  Hence $\A$ will not be weakly computably homogeneous.  Note that the isomorphism between the two orbits is still partial computable. That is, if we fix elements $a_1$ in the first orbit and $a_2$ in the second orbit
and map $a_1$ to $a_2$, then given $x$ in the orbit of $a_1$, we can compute $i \in \Z$ so that $x = f^i(a_1)$ and then compute $f^i(a_2)$ 
\end{proof}

Next,  we consider definability from exceptional sets for injection structures. 

\begin{proposition} \label{prop7}  Let $\A$ be a weakly ultrahomogeneous 
injection structure with no orbits of size one, and let $S=\{a_1,\ldots,a_n\}$ be a minimal
exceptional set.  Then $D(S)$ is the union of the finitely many
$\omega$-orbits of $\A$. 
\end{proposition}

\begin{proof}
Suppose that $a\in A$ is not in an $\omega$-orbit.  Then the map fixing all
$\omega$-orbits and sending $x\to f(x)$ otherwise is an automorphism
fixing $S$ but moving $a$. Now suppose $a$ is in an $\omega$-orbit
with $a_i$ for some $i\leq n$. Then for some $n\in\N$ we have
$f^{(n)}(a)=a_i$ or $f^{(n)}(a_i)=a$.  In either case, $a$ is
definable from $S$.
\end{proof}

If there is a unique orbit of size one, then of course this is actually definable in $\A$. 

Now we consider index sets for injection structures.  
The $e$th injection structure $\A_e = (\omega, \phi_e)$ is given by the
$e$th partial recursive function $\phi_e$ when $\phi_e$ is total and is an injection. It is easy to see that 
$INJ = \{e: \A_e \ \text{is an injection structure}\}$ 
is a $\Pi^0_2$ set. 

Let $UHI = \{e: \A_e\ \text{is an ultrahomogeneous injection structure}\}$ and let
$WUI= \{e: \A_e\ \text{is a weakly ultrahomogeneous equivalence structure}\}$.

\begin{theorem}
\begin{enumerate}
\item[(a)] The index set $UHI$ is $\Pi^0_2$ complete,
and in fact $\Pi^0_2$ complete relative to $INJ$. 
\item[(b)] The index set $WUI$  is $\Sigma^0_3$ complete,
and in fact $\Sigma^0_3$ complete relative to $INJ$. 
\end{enumerate}
\end{theorem}

\begin{proof} (a) Note that $\A_e$ is ultrahomogeneous if and only if $\phi_e$ is onto. It follows easily that $UHI$ is a $\Pi^0_2$ set.  For the completeness, we give a reduction $f$ of the $\Pi^0_2$ complete set $INF$ to
$UHI$ in such a way that $f(e) \in INJ$ for all $e$.    The idea of the construction is that $\A_{f(e)}$ will consist of a single infinite orbit, which will be of type $\Z$ if and only if $W_e$ is infinite.  The construction of $\phi_{f(e)}$ is in stages $\phi^s$ with domain $\{0,1,\dots,2s\}$ and image a subset of $\{0,1,\dots,2s+1\}$.  At stage 0, we let $\phi_{f(e)}(0) = 1$.  After stage $s$, we have defined a partial injection $\phi^s$ for all $i \leq 2s$.  Fix $s$ and let $a$ be the unique number $\leq 2s+1$ not in the image of $\phi^s$ and let $b$ the unique element $\leq 2s+1$ not in the domain of $\phi^s$.  There are two cases in the construction at stage $s+1$. If no new element comes into $W_e$ at stage $s+1$, extend $\phi^s$ by mapping $b$ to $2s+1$ and $2s+1$ to $2s+2$.  If a new element comes into $W_e$ at stage $s+1$, then again map $b$ to $2s+1$ but now map $2s+2$ to $a$.   

If $W_e$ is infinite, then it follows from the second case of the construction that $\A_{f(e)}$ will consist of a single orbit of type $\Z$, whereas if 
$W_e$ is finite, then $\A_{f(e)}$ will consist of a single orbit of type $\omega$, Thus $\A_{f(e)}$ is ultrahomogeneous if and only if $W_e$ is infinite.

(b) It is easy to see, by quantifying over the finitely many elements not in the range of $\phi_e$, that $WUI$ is a $\Sigma^0_3$ set. For the completeness, we give a reduction $g$ of the $\Sigma^0_3$ complete set $COF$ to $WUL$ in such a way that $g(e) \in INJ$ for every $e$. 
The idea of the construction is to build infinitely many orbits with the $i$th orbit consisting of $\{2^n(2i+1): n \in \omega\}$, so that the universe will be $\omega - \{0\}$.  The construction will ensure that the orbit of $2i+1$ is of type $\Z$ if $i \in W_e$ and otherwise is of type $\omega$. It follows that $\A_{g(e)}$ will be weakly ultrahomogeneous if and only if $W_e$ is cofinite.

The mapping $\phi = \phi_{g(e)}$ can simply be defined as follows. For each $i$, there are two cases. First, suppose that $i \notin W_e$.  Then $\phi(2^n(2i+1)) = 2^{n+1}(2i+1)$ for all $n$.  Next suppose that $i \in W_{e,m+1} - W_{e,m}$ for some $m$.  Then we let $\phi(2^n(2i+1)) = 2^{n+1}(2i+1)$ for $n < 2m$, we let $\phi(2^{2m+1}(2i+1)) = 2i+1$, and for all $n \geq 2m$, we let $\phi(2^{2n}(2i+1)) = 2^{2n+2}(2i+1))$ and $\phi(2^{2n+3}(2i+1)) = 2^{2n+1}(2i+1)$.  It follows that the orbit of $2i+1$ is of type $\omega$ if $i \notin W_e$ and is of type $\Z$ if  $i \in W_e$.  To compute $\phi_{g(e)}(2^n(2i+1))$, we just have to check whether $i \in W_{e,n+1}$ and then follow the algorithm above. 
\end{proof}

\section{Graphs}

It is a well-known theorem proved in \cite{LW80} that, up to isomorphism, the countable (both infinite or finite) ultrahomogeneous graphs are
\begin{enumerate}
\item The random graph, i.e. the Fraiss\'{e} limit of the class of all finite graphs
\item The $K_n$-free random graph, i.e. the the Fraiss\'{e} limit of the class of all $K_n$-free finite graphs.
\item $mK_n$, disjoint unions of $m$ copies of $K_n$ for $m,n\leq\aleph_0$.
\item The $3\times3$ lattice graph.
\item The cycle on 5 vertices.
\item Complements of these.
\end{enumerate}

Some examples of weakly ultrahomogeneous graphs which aren't ultrahomogeneous are:
\begin{itemize}
\item The disjoint union of an ultrahomogeneous graph and a finite graph.
\item (Equivalence relations) Any graph where all components are complete, there are only finitely many component sizes, and at most one size occurs infinitely often.
\end{itemize}

\begin{proposition} \label{prop8} Suppose $\G=(V,E)$ is a countable weakly ultrahomogeneous graph. Then $\G$ is a disjoint union of $\mathcal{H}$ and $mK_n$ for fixed $0\leq m,n\leq\aleph_0$ where $\mathcal{H}$ has finitely many connected components. 
\end{proposition}

\begin{proof}
We may assume that $\G$ has infinitely many components.  Let $\{a_1,\ldots,a_n\}\subseteq V$. If $\G$ has infinitely many non-ultrahomogeneous components, find one with no $a_i$ and use the non-ultrahomogeneity of the component to get a partial isomorphism that can't extend. If among the cofinitely many ultrahomogeneous components there are two kinds that occur infinitely often, find two non-isomorphic components containing no $a_i$. Choose $x$ from one and $y$ from the other, so $\la\textbf{a},x\ra\cong\la\textbf{a},y\ra$ but the isomorphism can't extend. Suppose the isomorphism type of the cofinitely many ultrahomogeneous components isn't $K_n$ for some $n$. Find two of these components containing no $a_i$.  Choose $x,y,z$ from these components such that $x,y$ are in the same component but don't share an edge and so $z$ is from the other component.  Then $\la\textbf{a},x,y\ra\cong\la\textbf{a},x,z\ra$ but it can't extend. 
\end{proof}

If for a graph $\G$, a subset of its components form a non-weakly ultrahomogeneous graph, then $\G$ itself is not weakly ultrahomogeneous. So, in light of Proposition \ref{prop8},  to classify weakly ultrahomogeneous graphs it is enough to look at graphs with finitely many infinite components.

Instead of a general classification, we examine a particular class of graphs. Say that a graph is locally finite if every vertex has finite degree, and say that a graph $\G$ is finitely dominated if there is a finite $F\subseteq\G$ so for every $v\in\G$ there is an $x\in F$ so $x\G v$. The proposition leads to a classification of locally finite weakly ultrahomogeneous graphs.

\begin{lemma} Let $\G$ be weakly ultrahomogeneous.
\begin{enumerate}[(a)]
\item $\G$ can have at most one component not finitely dominated.
\item If $\G$ has infinitely many components, every component is finitely dominated.
\item If $C$ is a component of $\G$ not finitely dominated, there must be a finite $F\subset C$ so every $v\in C$ has distance at most 2 from $F$.
\end{enumerate}
\end{lemma}

\begin{proof}
\begin{enumerate}[(a)]
\item Suppose $\G$ has 2 components $C_1$ and $C_2$ which are not finitely dominated. Let $A\subset\G$ be exceptional. Choose $x_1,y\in C_1$ and $x_2\in C_2$ so none of them share an edge with $A$ or with each other. Then $\la A, x_1,x_2 \ra\cong\la A,x_1,y\ra$, but the isomorphism can't extend since the component of $x_1$ is mapped to itself but $x_2$ is mapped to $y$.

\item  Suppose $\G$ has infinitely many components and a component $C$ not finitely dominated. Let $A\subseteq \G$ be exceptional. Choose $x_1,x_2\in C$ sharing no edge with $A$ or with each other. Then choose $y$ from a component of $\G$ containing no element of $A$. Then $\la A, x_1,x_2\ra\cong\la A,x_1,y\ra$, but the isomorphism can't extend.

\item  Suppose there is a component $C$ such that for every finite $F\subseteq C$, there is $v\in C$ with distance at least 3 from $F$. Let $A\subseteq \G$ be exceptional and let $v\in C$ have distance 3 from $A\cap C$ and $v'$ have distance 2. Then $\la A,v\ra\cong\la A,v'\ra$ but the isomorphism can't extend. Note that $A\cap C\neq\emptyset$ since $A$ must intersect each non-ultrahomogeneous component and any connected ultrahomogeneous graph has diameter 2. 
\end{enumerate}
\end{proof}

\begin{proposition} Let $\G$ be a locally finite graph. Then $\G$ is weakly ultrahomogeneous iff $\G=\mathcal{H}\cup mK_n$ where $0\leq m\leq\omega, 1\leq n<\omega$ and $\mathcal{H}$ is finite.
\end{proposition}
\begin{proof} Clearly a graph of this form is weakly ultrahomogeneous, so assume $\G$ is weakly ultrahomogeneous. If $\G$ has infinitely many components, then $\G=\mathcal{H}\cup\omega K_n$ where $n<\omega$ and $\mathcal{H}$ has finitely many components. By the lemma, every component of $\mathcal{H}$ is finitely dominated, so $\mathcal{H}$ is finitely dominated. But a finitely dominated, locally finite graph is finite.
\\Now assume that $\G$ has only finitely many components. At most on component, $C$, isn't finitely dominated. By the lemma there is a finite $F\subseteq C$ so every $v\in C$ has distance at most 2 from $F$. So $C$ is finite, along with all the other finitely dominated components, which means that $\G$ is finite.

\end{proof}

It follows that any locally finite, weakly ultrahomogeneous countable graph has a computable copy. 

Since graphs are relational structures, every weakly ultrahomogeneous graph is computably categorical.  But unlike for equivalence structures and linear orders, this containment is strict.
\begin{example} Let $\G$ be the following computable graph on $\omega$: for each $k$, let vertices $\{7k+1,7k+2,7k+3\}$ form a $K_3$ and let vertices $\{7k+4,7k+5,7k+6,7k+7\}$ form a star with center $7k+6$.  So the first three components are
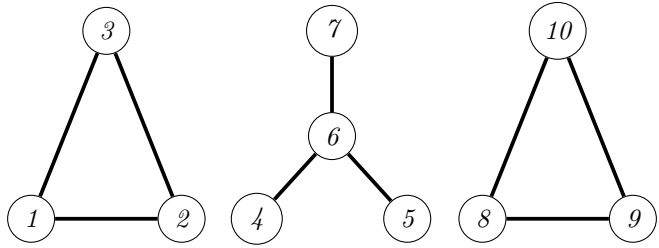
\begin{center}
\begin{tikzpicture}

[scale=.5,auto=left,every node/.style={circle,fill=black 20}]
  
  \node (n1) at (0,0)[circle,draw] {1};
  \node (n2) at (2,0)[circle,draw]  {2};
  \node (n3) at (1,2.5) [circle,draw] {3};
  
  \node (n4) at (3,0)[circle,draw]  {4};
  \node (n5) at (5,0)[circle,draw]  {5};
  \node (n6) at (4,1.1) [circle,draw] {6};
  \node (n7) at (4,2.5)[circle,draw]  {7};
  
    \node (n8) at (6,0)[circle,draw] {8};
  \node (n9) at (8,0)[circle,draw]  {9};
  \node (n10) at (7,2.5) [circle,draw] {10};

  \foreach \from/\to in {n1/n2,n2/n3,n3/n1,n4/n6,n5/n6,n7/n6,n8/n9,n9/n10,n10/n8}
    \draw [line width=.5mm] (\from) -- (\to);

\end{tikzpicture}
\end{center}
This graph can be shown to be computably categorical using a back-and-forth argument, but is not weakly ultrahomogeneous since infinitely many components aren't ultrahomogeneous.
\end{example}

\section{Trees}

The study of trees has played an important role in mathematical logic and computability theory. 
There are many ways to frame the notion of a tree. We are thinking in this paper of trees which are 
isomorphic to subtrees of $\omega^{<\omega}$.  Some terminology is necessary.   
 For a finite string $w = (w(0),w(1), \ldots, w(n-1))$, $|w|$ denotes the length $n$ of $w$. 
The empty string, denoted by $\epsilon$, is the unique string with length zero.
	 Given two strings $v$ and $w$, the \emph{concatenation} $v \fr w$ is defined by
\[
v \fr w = (v(0),v(1),\ldots,v(m-1),w(0),w(1),\ldots,w(n-1)),
\]
where $|v| = m$ and $|w| = n$.  For $m \leq |w|$, $w \restr m$ is the string $(w(0),\ldots,w(m-1))$. For any $X \in \omega^{\omega}$ and $n\in \omega$, 
the \emph{initial segment} $X \restr n$ is $(X(0),\ldots,X(n-1))$.
We say $w$ is an \emph{initial segment} or \emph{prefix} of $v$ (written $w \preceq v$) if $v = w \fr u$ for some $u\in \omega^{<\omega}$. 
This is equivalent to saying that $w = v \restr m$ for some $m\in \N$.  For a string $w$ and $X \in \omega^{\omega}$, we say that $w \prec X$ if $w = X \restr n$ for $n = |w|$.

A subset $T$ of $\omega^{<\omega}$ is a \emph{tree} if it  is closed under initial segments. That is, if $v \in T$ and $u \prec v$, then $u\in T$.
For any such tree $T$, $X\in \omega^{\omega}$ is said to be an \emph{infinite path through $T$} if $X \restr n \in T$ for all $n$. We let $[T]$ denote the set of infinite paths through $T$.  
For any tree $T$, we say that a node $w \in T$ is \emph{extendible} if there exists $X \in [T]$ such that $w \prec X$. 

A node $u \in T$ such that $v \notin T$ for any $v$ with $u \prec v$,  is called a \emph{dead end} or \emph{leaf} of $T$. 

More generally, trees can be formulated in terms of a natural partial ordering of $\prec$ described above, in terms of a binary infimum function,  in terms of successor functions, or in terms of a predecessor function. In the partial ordering formulation, we consider a tree $(T,\prec)$ as a partially ordered set with least element $\epsilon$ such that, for every $a \in T$,  $\prec$ well-orders the set $T(a) = \{x: x \prec a\}$; let the \emph{height} $ht_T(a)$ of $a$ in $T$ be defined as the order type of $T(a)$. The height of a tree $T$ is $ht(T) = sup\{ht_T(a): a \in T\}$. We will only consider trees of height $\leq \omega$, which means that $T(a)$ is finite for all $a$.  Thus $\prec$ will induce a natural predecessor function $f$, where $f(a)$ is the supremum of $T(a)$ under $\prec$, so that $f(\epsilon) = \epsilon$.  
There is also a natural \emph{meet} operation $\wedge$, defined by letting $u \wedge v$ be the supremum of $T(u) \cap T(v)$.    

In the predecessor formulation, we will consider a tree $(T,f)$ as a tree equipped with a unary predecessor function and possessing a unique 
\emph{root} $\epsilon$ such that $f(\epsilon) = \epsilon$ to make $f$ total.

For $a \in T$, let $T[a]$ denote the tree of extensions of $a$, that is, $T[a] = \{x: a \fr x \in T\}$. The \emph{rank} $rk_T(x)$ for $x \in T$ is defined by recursion as follows: 
\[
rk_T(x) = sup\{rk_T(y)+1: y \in T[x]\}
\]
In particular, for a leaf $x$ of $T$, $rk_T(x) = 0$. 
If $x$ is an extendible node of $T$, then $rk_T(x) = \infty$. Also, $rk(T) = rk_T(\epsilon)$.

R. Miller \cite{Mil05} showed that no computable tree of infinite height can be computably categorical as a partial ordering (or in the infimum
framework).  Lempp, McCoy, Miller and Solomon \cite{LMMS} characterized the computably categorical trees of finite height, in the partial ordering 
setting.  Calvert, Knight and J. Millar \cite{CKM} defined a notion of \emph{rank homogeneity} for trees in the predecessor formulation, where trees of infinite height can be computably categorical.  They constructed in particular a computable rank homogeneous tree of Scott rank $\omega_1^{ck}$. This notion was applied by Fokina, Knight, Melnikov, Quinn and Safranski \cite{FKMQS}, who showed that the class of rank homogeneous trees can be embedded into the class of torsion-free abelian groups and also into the class of Boolean algebras.  

We first consider trees of finite height, as partial orderings.   
Lempp, McCoy, R. Miller and Solomon \cite{LMMS}  characterized the computaby categorical trees $(T,\prec)$ of finite height using the notion of \emph{finite type}.

\begin{defn} A node $a$ of $T$ is \emph{of strongly finite type} if the set $S[a] = \{T[x]: x\ \text{is a successor of}\ a\}$ satisfies the following conditions: 
\begin{enumerate}
\item There are only finitely many isomorphism types in $S[a]$.
\item For any successors $x$ and $y$ of $a$, if $T[x]$ embeds into $T[y]$, then either $T[x]$ and $T[y]$ are isomorphic, or the isomorphism type of $T[y]$ appears only finitely often in $S[x]$.
\end{enumerate}
\end{defn}

$T$ is of strongly finite type if every node of $T$ is of strongly finite type. The notion of finite type is given by a recursive definition as follows. 

\begin{defn} A node $a$ of $T$ is \emph{of finite type} if it satisfies the following: 
 the set $S[a] = \{T[x]: x\ \text{is a successor of}\ a\}$ satisfies the following conditions: 
\begin{enumerate}
\item There are only finitely many isomorphism types in $S[a]$.
\item Each isomorphism type which appears infinitely often in $S[a]$ is of strongly finite type.
\item For any successors $x$ and $y$ of $a$, if $T[x]$ embeds into $T[y]$, then either $T[x]$ and $T[y]$ are isomorphic, or the isomorphism type of $T[x]$ appears only finitely often in $S[x]$  or the isomorphism type of $T[y]$ appears only finitely often in $S[x]$.
\end{enumerate}
\end{defn}

$T$ is of finite type if every node in $T$ is of finite type. Lempp et al \cite{LMMS} show that $(T,\prec)$ is computably categorical if and only if $(T,\prec)$ is of finite type. 

It is easy to see that a tree $(T,\prec)$ is ultrahomogeneous if and only if it has rank $\leq 1$,  and likewise for the meet presentation. Note that this is equivalent to saying that $T$ has height $\leq 1$. It follows from the result of Miller that if $(T,\prec)$ is computably categorical, then it has finite height. Since trees are locally finite, every weakly homogeneous tree must be computably categorical by Theorem \ref{wmain}, and hence any weakly homogeneous tree $(T,\prec)$ must have finite height. We give a short direct proof of this fact here.

\begin{lemma} \label{lemwfh}  If $(T,\prec)$ is weakly homogeneous, then $T$ has finite height.
\end{lemma}

\begin{proof} Suppose that $T$ has infinite height and let a finite set $S$ be given.  Then let $n$ be the maximum height of any element of $S$,  let $b$ be an element of height $>n$ of rank $\geq 1$ and let $c$ be a successor of $b$.  Then there is an isomorphism of $S \cup \{b\}$ to $S \cup \{c\}$ fixing $S$ and mapping $b$ to $c$, which cannot be extended to an automorphism of $T$.  Hence $T$ is not weakly ultrahomogeneous. 
\end{proof}

\begin{prop} \label{prop9}  $(T,\prec)$ is weakly ultrahomogeneous if and only if the set of elements which have rank $\geq 1$ is finite. 
\end{prop}

\begin{proof} Suppose  that $T$ is weakly ultrahomogeneous but has infinitely many elements of rank $\geq 1$.  It follows from Lemma \ref{lemwfh} that $T$ has finite height. Now let $n$ be the least such that there are infinitely many elements with height $n$ which have rank $\geq 1$. Now given a finite set $S$, choose an element $b$ of height $n$ and rank $\geq 1$, which is not in the downward closure of $S$ and let $c$ be a successor of  $b$. Then there is an isomorphism of $S \cup \{b\}$ to $S \cup \{c\}$ fixing $S$ and mapping $b$ to $c$, which cannot be extended to an automorphism of $T$. Hence $T$ is not weakly ultrahomogeneous.

For the other direction, suppose that all but finitely many elements of $T$ have rank 0, that is, are leaves of $T$. Let $S$ be the set of elements which have successors, 
and suppose that $\phi$ is an isomorphism from a subtree $T_1$ of $T$ to a subtree $T_2$ of $T$, where $S \subseteq T_1$ and $S \subseteq T_2$ and $\phi(u) = u$ for all $u \in S$.  For any element $x$ of $T_1 - S$, $x$ has a predecessor $a$ in $S$ and has no successors. Then $\phi(x)$ must also be a successor of $a$ and have no successors. 
It follows that any permutation of $S[a]$ will be an isomorphism, so that $\phi$ can be extended to an automorphism of $T$. 
\end{proof}

It follows that every weakly ultrahomogeneous tree $(T,\prec)$ has a computable copy.  That is, $T$ has a finite subtree $S$ of nodes which have successors and each element of this finite tree has some (possibly infinite) number of successors, which are all leaves of $T$ and this clearly always has a computable representation. 

If $T$ has only finitely many elements of height $\geq 1$, then certainly it has only finitely many elements of rank $\geq 1$, and is therefore weakly ultrahomogeneous. 
The following example shows that these are not equivalent conditions. 

\begin{example} Let $T = \{\epsilon,(0)\} \cup \{(0,m): m \in \omega\}$. Then only the first two elements have rank $\geq 1$, but infinitely many elements have height 2,
and $T$ is weakly ultrahomogeneous with exceptional set $\{\epsilon, (0)\}$. 
\end{example}

The following result shows that the weakly ultrahomogeneous trees live well inside the class of computably categorical trees.

\begin{prop} Every computable weakly ultrahomogeneous tree $(T, \prec)$  is of strongly finite type. 
\end{prop}

\begin{proof} Let $(T,\prec)$ be a computable weakly ultrahomogeneous tree. Every leaf node is clearly of strongly finite type, so we consider a node $a\in T$ such that all nodes with rank less than $rk(a)$ are of strongly finite type. There are only finitely many isomorphism types of $T[x]$ in the whole tree, hence only finitely many in $S[a]$ and each is of strongly finite type by assumption. Among the successors of $a$ all but finitely many are leaf nodes, so the second condition will be satisfied. Thus $a$ is of strongly finite type.
\end{proof}

We can characterize exceptional sets for weakly homogeneous trees $(T,\prec)$ as follows. For any $K \subseteq T$, let $L_{K,T}(a) = \{x \in K: x \preceq a\}$ and let 
$U_{K,T}(a) = \{x\in K: a \preceq x\}$. 

\begin{lemma} \label{lemluiso}  For any tree $(T, \prec)$ and any $a,b \in T$, $K \cup \{a\}$ is isomorphic to $K \cup \{b\}$ if and only if $L_{K,T}(a) = L_{K,T}(b)$ and
$U_{K,T}(a) = U_{K,T}(b)$.
\end{lemma}

\begin{proof} $K \cup \{a\}$ is isomorphic to $K \cup \{b\}$ if and only if, for any $x \in K$, $x \preceq a \iff x \preceq b$ and $a \preceq x \iff b \preceq x$, which is if and only if $L_{K,T}(a) = L_{K,T}(b)$ and
$U_{K,T}(a) = U_{K,T}(b)$.
\end{proof}

\begin{prop} \label{prop9.5} Let $(T, \prec)$ be weakly ultrahomogeneous. Then for any $K \subseteq T$,
$K$  is an exceptional set for $(T,\prec)$ if and only if
\begin{itemize}
\item[(i)] For $a \in T$ of rank $\geq 1$, there exists $y \in K$ such that $a \preceq y$, and 
\item[(ii)] For any $a,b \in T$ such that $a \prec b$, either $b \in K$ or there exists $z \in K$ such that $a \preceq z$ but not $b \preceq z$.  
\end{itemize}
\end{prop}

\begin{proof} Supose first that $S$ does not satisfy condition (i) above, and let $a \prec b$ where there is no element of $K$ above $a$.  Then  $U_{K,T}(a) = U_{K,T}(b) = \emptyset$ and 
 $L_{K,T}(a) = L_{K,T}(b)$ since $a \prec b$ and there is no element of $K$ above $a$. It follows from Lemma \ref{lemluiso} that $K  \cup \{a\}$ and $K \cup \{b\}$ are isomorphic.
This cannot be extended to an automorphism of $T$ since the height of $a$ is less than the height of $b$.  Next suppose that $S$ does not satisfy condition (ii) and let $a \prec b$ such that $b \notin K$ and for any $x \in K$, if $a \prec x$, then $b \prec x$.  Again it follows that $U_{K,T}(a) = U_{K,T}(b)$ and $L_{K,T}(a) = L_{K,T}(b)$ , so that $K \cup \{a\}$ and $K \cup \{b\}$ are isomorphic,  but there is no extension of this isomorphism to an automorphism of $T$. 

For the other direction, suppose that $K$ satisfies the two conditions.  We claim that for any two distinct elements $a$ and $b$ of rank $\geq 1$ in $T$, either $U_{K,T}(a) \neq U_{K,T}(b)$ or $L_{K,T}(a) \neq L_{K,T}(b)$. The proof of the claim is in two cases. 

First, suppose that $a$ and $b$ are incomparable. Then by condition (i), there exists an element $x \in K$  such that $a \preceq x$ but we cannot have $b \prec x$ since $T$ is a tree. 
Thus $U_{K,T}(a) \neq U_{K,T}(b)$.  Next, suppose without loss of generality that $a \prec b$. By condition (ii), there are two possibilities. We may have $b \in K$, in which case $b \in L_{K,T}(b) \setminus L_{K,T}(a)$.  Or we have  $z \in K$ such that $a \preceq z$ but not $b \preceq z$.  In that case, $z \in U_{K,T}(a) \setminus U_{K,T}(b)$.  This proves the claim.

Now let $\phi$ be an isomorphism mapping $a_1,\dots,a_k$ to $b_1,\dots,b_k$ and fixing each element of $K$. For each $a_i$ of rank $\geq 1$, we have some $x \in K$ such that 
$a_i \preceq   x$ and thus by the isomorphism $b_i \preceq x$. Thus $b_i$ also has rank $\geq 1$. The isomorphism $\phi$ shows that $K \cup \{a_i\}$ is isomorphic to $K \cup \{b_i\}$, so that 
$L_{K,T}(a_i) =  L_{K,T}(b_i)$ and $U_{K,T}(a_i) = U_{K,T}(b_i)$. It now follows from the claim that in fact $a_i = b_i$.  Thus we may define our desired automorphism $H$  to be the identity on $K$ and on the subtree of nodes of rank $\geq 1$. To complete the definition of $H$, it suffices to define $H$ on the leaves of $T$.  Fix an element $c$ of $T$  with a nonempty set $L[c]$ of leaves in $S[c]$. For any $x \in L[c] \cap K$, let   $H(x) =  x$. For the remaining elements of $L[c]$, observe that any permutation of these elements may be used to obtain an automorphism of $T$. So we may use any permutation for $H$ which agrees with $\phi(a_i)$ for each $a_i \in L[c]$.
\end{proof}

Now we consider index sets for trees under the partial order presentation.  The $e$th tree $\A_e = (\omega, \prec_e, \epsilon)$ is given by the
$e$th partial recursive function $\phi_e$ when $\phi_e$ is total and is the characteristic function of a tree.  The usual condition for a partial ordering to be a tree is that, for any element $a$, $\{x: x \prec_e a\}$ is well-ordered.  This would be a $\Pi^1_1$ condition, but it can be simplified here by first requiring that $\{x: x \prec_e a\}$ is finite, which is a $\Pi^0_3$ condition, and then checking to see that  it is totally ordered.  It follows that $TRO= \{e: \A_e \ \text{is a tree}\}$ is a $\Pi^0_3$ set. Let   $UHT = \{e: \A_e\ \text{is an ultrahomogeneous tree}\}$ and let $WUT= \{e: \A_e\ \text{is a weakly ultrahomogeneous tree}\}$. 

\begin{theorem}
\begin{enumerate}
\item[(a)] The index set $TRO$  is $\Pi^0_3$ complete.
\item[(b)] The index set $UHT$ is $\Pi^0_1$ complete relative to $TRO$. 
\item[(c)] The index set $WUT$ is $\Sigma^0_2$ complete relative to $TRO$. 
\end{enumerate}
\end{theorem}

\begin{proof} (a) We will give a reduction $h$ of the $\Sigma^0_3$ complete set $COF$ to the complement of $TRO$ as follows. The idea of the construction is to build a descending chain below $2n+1$
 if and only if  $m \in W_e$ for all $m \geq n$.  Thus if $W_e$ is cofinite, $T_{h(e)}$ will fail to be a tree; the construction will ensure that it is a tree otherwise.  First let the root $\epsilon= 0$ and for ease of notation let $T = T_{h(0)}$.  After stage $s$, we have a finite tree $T^s = \{0,1,\dots,2s\}$ with $0 \prec m+1$ for all $m \leq 2s$ and such that, for each $n < s$, if  $n,n+1,\dots,n+k \in W_{e,s}$ for some $k$, then there is a finite decreasing chain $2n+1, 2 i_1, \dots, 2 i_k$ in $T^s$.  At stage $s+1$, we first add the elements $2s+1, 2s+2$ to $T^s$ and let $0 \prec 2s+1, 2s+2$.  Next suppose that a new element $m$ comes into $W_e$ such that, for some $j$, $k$ and $n$,  $n,n+1,\dots,n+j-1 = m-1 \in W_{e,s}$ with a corresponding chain $2n+1, 2 i_1, \dots, 2i_{h-1}$ in $T^s$ and $m+1 = n+j+1,n+j+2,\dots,n+j+k-1 \in W_{e,s}$ with corresponding chain $2m+3,2p_1,\dots,2p_{k-1}$. Then we insert  $2s+2$ between these two chains, resulting in a descending chain from $2n+1$ of length $j+k-1$.  Note that if $j=0$, then the first chain is empty and if $k=0$, then the second chain is empty.  It follows that if $W_e$ is cofinite and includes $\{n,n+1,\dots\}$, then $T_{h(e)}$ will include an infinite descending chain and hence will not be a tree. On the other hand,
if $W_e$ is co-infinite, then for each $a$, $\{x: x \prec a\}$ will be finite, and will be totally ordered by the construction, so that $T_{h(e)}$ will be a tree.

 (b)  $\A_e$ is ultrahomogeneous if and only if there do not exist distinct $a$ and $b$, both not the root, such that $ a\prec_e b$.  It follows easily that $UHT$ is a $\Pi^0_1$ relative to $TRO$.  For the completeness, we define a computable function $f$ such that $\A_{f(e)}$ has an element of height 2 if and only if $W_e$ is nonempty. Note that $\{e: W_e = \emptyset\}$ is $\Pi^0_1$ complete.  To define $\A_{f(e)}$, simply let $0$ be the root, let $0 \prec s$ for all $s$ and, for $s>0$,  let $s \prec t$ if and only if $W_{e,t}$ is nonempty.  

(c)   $\A_e$ is weakly ultrahomogeneous if and only if there is a finite set of elements of rank $\geq 1$ and a $\Sigma^0_2$ formula may be given by quantifying over these finite sets. 
For the completeness, we define a reduction $g$ of the $\Sigma^0_2$ complete set $FIN  = \{e: W_e\ \text{is finite}\}$ to $WUT$ so that $g(e) \in TRO$ for all $e$. 
Again let $0$ be the root, let $0 \prec_{g(e)} x$ for all $x$ and let $2s+1 \prec_e 2s+2$ if an element comes into $W_e$ at stage $s$.  It is clear that $\A_{g(e)}$ is always a tree, and  will have finitely many nodes of rank $\geq 1$ if and only if $W_e$ is finite. 
\end{proof}

Next we will consider trees under the predecessor formulation. The notions of ultrahomogeneity turn out to be more complicated here. 

First an easy implication connecting the two formulations. 

\begin{prop} For any tree $T$,  if $(T,\prec)$ is (weakly) ultrahomogeneous, then $(T,f)$ is (weakly) ultrahomogeneous. Similarly, if $(T,\prec)$ is computably categorical, then $(T,f)$ is computably categorical.
\end{prop}

\begin{proof}  In general,  $(T,\prec)$ is relational, and $(T,f)$ is locally finite, so in either case every weakly homogeneous tree is computably categorical.  Assume that $(T,\prec)$ is weakly ultrahomogeneous and let $T_1$ and $T_2$ be finite isomorphic subtrees of $(T,f)$ including some fixed finite set $S$.  Then there is an isomorphism between $(T_1,\prec)$ and $(T_2,\prec)$ fixing $S$. Thus there is an automorphism of $(T,\prec)$ extending  this isomorphism and this is also an automorphism of $(T,f)$. Next assume that $(T,\prec) $ is computably categorical and that $(T,f)$ is computable, and let $(S,f)$ be a computable tree isomorphic to $(T,f)$. Then $(S,\prec)$ is also computable, and it follows that there is a computable isomorphism mapping $(T,\prec)$ to $(S,\prec)$ which also serves as an automorphism of $(T,f)$. 
\end{proof}

 In both cases, the converse fails to hold. For example, consider the tree $T = \{0^n: n \in \omega\}$ with $f(0^{n+1}) = 0^n$; this tree is homogeneous in the predecessor forumulation but has infinite height and is not computably categorical in the partial order presentation. It is important to note this distinction, that we will be considering trees of possibly infinite height in the predecessor formulation, whereas such trees could not be even  weakly ultrahomogeneous in the partial order formulation. 

\begin{theorem} \label{prop10} A tree $(T,f)$ with predecessor function is ultrahomogeneous if and only if, for every $n$ and any $a,b \in T$ of the same height, $a$ and $b$ have an equal number of successors. 
\end{theorem}

\begin{proof} Suppose first that for every any $a,b \in T$ of the same height have an equal number of successors. 
Then it is clear that automorphism of $T$ may be defined recursively by first permuting the elements of height $1$ and then 
arbitrarily mapping successors of $x$ to successors of $\phi(x)$ for $x \in T$ of height $n$ to define $\phi$ on elements of height $n+1$. Now given finite
isomorphic subtrees $U_1$ and $U_2$  with all elements of height $\leq n$, the isomorphism $\phi$ may first be extended to  an isomorphism on the elements of height $\leq n$ , since there will always be an equal number of elements of any height which have not been mapped yet. 
Then the  isomorphism may be recursively extended to an automorphism $\Phi$ of $T$ 
by mapping a successor $x'$ of $x$ to a successor $y'$ of $\Phi(x)$.  

Suppose next that there is some $n$ and some elements $a,b \in T$ of height $n$ which have a different number of successors, where $n$ is the least for which such elements exist. Then there is an isomorphism  of the elements of height $\leq n$ mapping $a$ to $b$ by the argument above. But this isomorphism cannot be extended to an automorphism of $T$.
\end{proof}

Note that under the condition of  Proposition \ref{prop10}, there must be a function $\beta: \N \to \N \cup \{\omega\}$ such that every node of height $n$ has exactly $\beta(n)$ immediate successors.  An equivalent condition to saying that any two  nodes of height $n$ have the same number of successors,  is to say that for any $n$ and any two nodes $a$ and $b$ of height $n$, $T[a]$ is isomorphic to $T[b]$.

{\bf  Remark}: It follows from the proof of Proposition \ref{prop10} that if $T_1$ and $T_2$ are ultrahomogeneous with the same branching function $\beta$, then any isomorphism taking a finite subtree of $T_1$ to a finite subtree of $T_2$ can be extended to an isomorphism from $T_1$ to $T_2$. 

Given a tree $T$ (of possibly infinite height) and a subtree $U$ of finite height with node $a \in U$, let $T_U[a] = \{a\} \cup T(a) \cup  \bigcup \{T[x]: x \in T[a] \setminus U\}$. 
For example, if $T = \omega^{<\omega}$ and $U = \{x: x(0) > 1\}$, then $T_U[\epsilon] = \{\epsilon\} \cup \{x: x(0) \leq 1\}$ and $T_U[(2)] = \{\epsilon,(2)\}$. So $T_U[a]$ consists of the node $a$ and its predecessors, together with all nodes in $T$ extending successors of $a$ which are not in $U$.

\begin{theorem} \label{pwuh} A tree $(T,f)$ in the predecessor framework is weakly ultrahomogeneous if and only if there is a finite subtree $S$ of $T$ such that, for every $x \in S$, $T_S[x]$ is ultrahomogeneous. 
\end{theorem}

\begin{proof}   Suppose first that $(T,f)$ is weakly ultrahomogeneous and let $S$ be a finite exceptional tree.  Let $x \in S$ and let $\phi$ be an isomorphism between two finite subtrees $U_1$ and $U_2$ of $T_S[x]$; extend $\phi$ to the other elements of $S$ (in particular the immediate successors of $a$ which are in $S$), by the identity  and let $T_1 =  S \cup U_1 \cup \{y: y \prec x\}$ and  $T_2 = S \cup  U_2 \cup \{y: y \prec x\}$.  Then $T_1$ and $T_2$ are finite subtrees of $T$ and, since $S$ is an exceptional set, $\phi$ may be extended to an automorphism of $T$ which fixes $x$. It is clear that the restriction of this automorphism to $T[x]$ is the desired automorphism of $T_S[x]$. 

Next suppose that $S$ is a finite subtree as specified; we will show that $S$ is an exceptional set.  Let $\phi$ be an isomorphism of two finite subtrees $R_1$ and $R_2$ of $T$ which fixes $S$. It follows that, for each $x \in S$, $\phi(x) = x$ and hence the map induced by $\phi$ on $T[x]$ is an isomorphism between $R_1[x]$ and   $R_2[x]$.  By assumption, there exist, for each $x \in S$, an automorphism $\phi_x$ on $T[x]$ which extends the map between $R_1[x]$ and $R_2[x]$. 
These can be patched together to define an automorphism $H$ of $T$ which preserves $S$ and extends $\phi$. That is, $H(x \fr y) = \phi_x(y)$ when $x \fr y \notin S$ and equals $x \fr y$ otherwise. 
\end{proof}

Let us consider this further for trees of height $\leq 3$. 

\begin{proposition} \label{2wuht}
A tree of height $\leq 2$ is weakly ultrahomogeneous if and only if all but finitely many nodes of height 1 have an equal number of successors.
\end{proposition}

\begin{proof} Given $k$ such that all but finitely many nodes of height 1 have exactly $k$ successors, let $S$ be the subtree consisting of the root $\epsilon$ together with the set of nodes of height 1 which have a different number of successors. Then $T_S[\epsilon]$ consists of $\epsilon$ together with all nodes of height 1 having exactly $k$ successors and each of their $k$ successors. Thus $T_S[\epsilon]$ is ultrahomogeneous.  For $x \in S$ of height 1, $T_S[x] = S[a]$, that is a tree of height 1 consisting of $x$ together with all of its successors which is trivially ultrahomogeneous.

For the other direction, suppose that $T$ is weakly ultrahomogeneous, and let $S$ be given by Theorem \ref{pwuh} so that $T_S[x]$ is ultrahomogeneous for all $x \in S$.   Then $T_S[\epsilon]$ will include each node $x$ of height 1 not in $S$ and having no successor in $S$, together with all of their successors. Since $T_S[\epsilon]$ is ultrahomogeneous, it follows that all nodes of height 1 not in $S$ will have the same number of successors, as desired. 
\end{proof}

\begin{proposition} \label{3wuht}
A tree of height 3 is weakly ultrahomogeneous if and only if the following conditions hold: 
\begin{itemize}
\item[(a)] for each node $x$ of height 1, all but finitely many successors of $x$ have an equal number of successors;
\item[(b) ]there are fixed $h$ and $k$ in $\omega \cup \{\omega\}$ such that all but finitely many nodes of height 1 have exactly $h$ successors and each of those successors has exactly $k$ successors. 
\end{itemize}
\end{proposition} 

\begin{proof} Assuming conditions (a) and (b), let $S$ consist of $\epsilon$ together with the following: Let $x$ be a node of  height 1 
which does not have exactly $h$ successors, or which has exactly $h$ successors but some of those successors do not have exactly $k$ successors. Then,  by (a), fix $m$ such that all but finitely many successors of $x$ have exactly $m$ successors; then put $x$ into $S$  together with each of its successors which do not have exactly  $m$ successors. Then $T_S[\epsilon]$ consists of $\emptyset$ together with all nodes $x$ of height 1 such that $x$ has exactly $h$ successors and each of these has exactly $k$ successors. Thus $T_S[\epsilon]$ is ultrahomogeneous. 

 For $x \in S$ of height 1 such that all but finitely many successors of $x$ have exactly $m$ successors, $T_S[x]$ consists of $x$ together with each successor which has exactly $m$ successors, and again $T_S[x]$ is ultrahomogeneous. 

For the other direction, suppose that $T$ is weakly ultrahomogeneous, and let the finite subtree $S$ be given by Theorem \ref{pwuh} so that $T_S[x]$ is ultrahomogeneous for all $x \in S$. Then $T_S[\epsilon]$ will contain each node $x$ of height 1 with no extensions in $S$, together with all extensions of $x$.  Since $T_S[\epsilon]$ is ultrahomogeneous, it follows from Proposition \ref{prop10} that all nodes of height 1 not in $S$ will have the same number ($h$) of successors,and that each of these successors will have the same number ($k$) of successors.  Next consider $x \in S$ of height 1. Then $T_S[x]$ will contain all successors $y$ of $x$ which are not in $S$, together with all extensions of $y$.  Since $T_S[x]$ is ultrahomogeneous, it follows that all successors of $x$ not in $S$ will have an equal number of successors, as desired. 
\end{proof}

Certainly there are computably categorical trees,  in either presentation,  which are not weakly ultrahomogeneous.

\begin{example}
Let $T$ have infinitely many nodes of height 1 with exactly 2 successors and infinitely many with exactly 3 successors, and then let each node of a pair of successors  have exactly 4 successors and each node of a triple of successors have exactly 1 successor. It can be checked that this tree is of strongly finite type and is therefore computably categorical.  On the other hand, $T$ is not weakly ultrahomogeneous in either presentation. 
\end{example}

As for injection structures, there are continuum many ultrahomogeneous trees $(T,f)$.

\begin{proposition} 
 For any function $\beta: \omega \to \omega \cup\{\omega\}$, there is an ultrahomogeneous tree $(T,f)$ with branching function $\beta$ and furthermore $(T,f)$ is relatively computably categorical.  
\end{proposition}

A tree $T$ is said to be \emph{rank-homogeneous} \cite{CKM} if it satisfies the following conditions for all $n \in \omega$ and all $a \in T$ of height $n$:
\begin{enumerate}
\item for all ordinals $\sigma < \alpha=rk(a)$, if some $b$ of height $n+1$ has rank $\sigma < \alpha$, then $a$ has infinitely many successors of rank $\sigma$. 
\item If $rk(a) = \infty$, then $a$ has infinitely many successors of rank $\infty$.
\end{enumerate}

Let $R_n(T) = \{rk(a): height(a) = n\}$. The following results can be found in \cite{CKM}.

\begin{prop} [CKM] Suppose $T$ and $T'$ are  rank homogeneous trees such that $R_n(T) = R_n(T')$ for all $n$. Then $T$ and $T'$ are isomorphic. 
\end{prop}

\begin{prop}[CKM] If $T$ is a rank-homogeneous tree, then for any tuples $\overbar{a} = (a_1,\dots,a_k)$ and $\overbar{b} = (b_1,\dots,b_k)$, there is an automorphism of $T$ taking $\overbar{a}$ to $\overbar{b}$ if and only if the function taking to $\overbar{a}$ to $\overbar{b}$ extends to a rank-preserving  isomorphism from the finite subtree generated by $\overbar{a}$ to the finite subtree generated by $\overbar{b}$.
\end{prop}

Despite this close analogy with the notion of ultrahomogeneity, rank-homogeneous trees are not necessarily ultrahomogeneous. For example, consider the tree $T$ containing all nodes $(n)$ of length 1 and such that $(2^n \cdot (2m+1))$ has exactly $n$ immediate successors for each $m$ and $n$. Then $T$ is certainly rank homogeneous, but is not weakly ultrahomogeneous.

On the other hand, consider the tree $T$ which has only two nodes $(0)$ and $(1)$ of height 1 and then has all possible successors of these. Then $T$ is ultrahomogeneous but is not rank-homogeneous.

Finally we consider index sets for trees under the predecessor presentation.  The $e$th tree $\B_e = (\omega, \phi_e, \epsilon)$ is given by the
$e$th partial recursive function $\phi_e$ when $\phi_e$ is total and $\phi_e$ is the predecessor function of a tree.   For a tree with  predecessor function $f$, for any element $a$,
$\{x: x \prec a\} = \{f^n(a): n \in \omega\}$, and this set is ordered by having $f^m(a) \prec  f^n(a) \iff m > n$.  Hence $(\omega,f)$ is a tree provided that 

\[
(\forall a)(\exists n)[ f^n(a) = \epsilon]\ \&\ (\forall a)(\forall m \neq n)[f^m(a) = f^n(a) \ \Longrightarrow\ f^n(a) = \epsilon].
\]

It follows that $TRP = \{e: \A_e \ \text{is a tree}\}$ is a $\Pi^0_2$ set. 

 Let $UHP= \{e: \A_e\ \text{is an ultrahomogeneous tree}\}$, and let 
$WUP= \{e: \A_e\ \text{is a weakly ultrahomogeneous tree }\}$.

\begin{theorem}
\begin{enumerate}
\item[(a)] The index set $UHP$ is $\Pi^0_2$ complete relative to $TRP$. 
\item[(b)] The index set $WUP$ is $\Sigma^0_3$ complete relative to $TRP$. 
\end{enumerate}
\end{theorem}

\begin{proof} (a)  $\B_e$ is ultrahomogeneous if and only any two elements of the same height have an equal number of successors. It is computable to test whether $a$ and $b$ have the same height, by computing $f^n(a)$ and $f^n(b)$ for $n = 1,2,\dots$ and checking that the least $n$ such that $f^n(a) = \epsilon$  is also the least $n$ such that $f^n(b) = \epsilon$.  Now $a$ has $\geq k$ successors if and only if there exist $k$ distinct elements $x_1,\dots,x_k$ such that $\phi_e(x_i) = a$ for each $i$, which makes this a $\Sigma^0_1$ relation of $a,k$.
Next we see that $a$ and $b$ have an equal number of successors if and only if, for each $k$, $a$ has $\geq k$ successors if and only if $b$ has $\geq k$ successors, making this a $\Pi^0_2$ relation. Quantifiying over $a$ and $b$ of the same height, we obtain a $\Pi^0_2$ characterization of ultrahomogeneity.

For the completeness, we define a computable functrion $g$ such that $\B_{g(e)}= \B$ is ultrahomogeneous if and only if $W_e$ is infinite, and is always a tree.  Let $0$ be the root of $\B_{g(e)}$ and have immediate successors $1$ and $2$.  Then at each stage $s+1$, give $1$ an additional successor, and give $2$ an additional successor if and only if a new element enters $W_e$.  Then $1$ and $2$ have the same height, and $1$ has infinitely many immediate successors, but $2$ will have infinitely many immediate successors if and only if $W_e$ is infinite.

(b)  $\B_e$ is weakly ultrahomogeneous if and only if there is a finite subtree $T$ of $\B$ such that $B_T[x]$ is ultrahomogeneous for each $x \in T$; this gives a $\Sigma^0_3$ form. 
For the completeness, we define a reduction $h$ of the $\Sigma^0_3$ complete set  $COF  = \{e: W_e\ \text{is cofinite}\}$ to $WUT$ so that $h(e) \in TRP$ for all $e$. 
Again let $0$ be the root, but now let $0$ have infinitely many immediate successors, $2^n$, for each $n$.   For each $n$, 
the construction will give the node $2^n$ $n-1$ immediate successors, and in addition a total of $n+k$ immediate successors if and only if $n+1,n+2,\dots,n+k$ all belong to $W_e$.  
Thus if $W_e$ is cofinite and contains $\{n+1,n+2,\dots\}$, then the nodes $2^n,2^{n+1},\dots$ will all have infinitely many successors. The tree $T$ can consist of the root together with the finitely many nodes $2^n$ which have only finitely many successors. Note that in $B_T[0]$ every immediate successor of $0$ has infinitely many extensions, so that $B_T[0]$ is ultrahomogeneous. 

If $W_e$ is co-infinite, then the nodes of the form $2^n$ all have finitely many immediate successors, but the number of successors will grow with $n$, so that for any finite subtree $T$,
$B_T[0]$ will not be ultrahomogeneous. 
\end{proof}

\section{$n$-Equivalence Structures}

In this section, we study a generalization of equivalence structures allowing for more than one equivalence relation on the universe.

\begin{defn} For $n<\omega$, an $n$-equivalence structure is a structure $\A=(A,E_1,\ldots,E_n)$ where each $E_i$ is an equivalence relation on $A$.  An $n$-equivalence structure is nested if for $i<j\leq n$ we have $ x E_j y\rightarrow x E_i y$, i.e $E_j \subseteq E_i$ as subsets of $A\times A$.  For $a \in A$,  we let $[a]_i$ denote the equivalence class of $a$ under $E_i$. Thus for a nested equivalence structure, $i < j \leq n$ implies that $[a]_j \subseteq [a]_i$,  so that the $E_i$ classes are partitioned by $E_j$. There is an implicit equivalence relation $E_0 = A \times A$, so that $[a]_ 0 = A$ for all $a$. \end{defn}

For example, consider the relation on a given familly of structures defined by $\A \equiv_n \B$ if and only if $\A$ and $\B$ satsify the same 
sentences of quantifier rank $n$. This plays an important role in mathematical logic. 

It is easy to see that if an $n$-equivalence structure $\A$ is ultrahomogeneous, then each individual equivalence structure $(A,E_i)$ must be ultrahomogeneous.
In general, this condition is not sufficient. It must at least be the case that for any $a$ and $b$, the intersections $[a]_i \cap [a]_j$ and $[b]_i \cap [b]_j$ have the same cardinality. For example, let $A = \{1,2,3,4,5,6\}$ have $E_1$ classes $\{1,2,3\}$ and $\{4,5,6\}$ and $E_2$ classes $\{1,2\}$, $\{3,4\}$ and $\{5,6\}$.
Then $(A,E_1)$ and $(A,E_2)$ are ultrahomogeneous, but $(A,E_1,E_2)$  is not ultrahomogeneous since $[1]_1 \cap [1]_2 = \{1,2\}$ but 
$[4]_1 \cap [4]_2 = \{4\}$.  The other direction for nested structures is considered below.  

In \cite{Marshall}, Leah Marshall describes an effective correspondence between nested $n$-equivalence structures and certain trees of finite height where the branching of the tree reflects the containment of equivalence classes. This correspondence allows many effective properties to be transferred between nested $n$-equivalence structures and trees of finite height.
 
\begin{defn} \label{defTA} For any $n$-equivalence structure $\A = (A,E_1,\dots,E_n)$, let $E_0 = A \times A$, let $E_{n+1}$ be equality, and define the tree $T_{\A}$  as follows. The universe of $T_{\A}$ is the set $\{[a]_i: a \in A, i = 1,\dots,n\}$ and the partial ordering is inclusion. This means that for each $a$ and $i \leq n$, $[a]_i$ is the predecessor of $[a]_{i+1}$. 
\end{defn}

Marshall  shows that a representation of $T_{\A}$ can be computed from $\A$ so that the mapping from $a$ to $[a]$ is also computable from $\A$. 
 Recalling the definition of trees of finite type from section 7, here is a key result of \cite{Marshall}.

\begin{theorem}  [Marshall \cite{Marshall}] Let $\A$ be a computable $n$-equivalence structure and $T_{\A}$ its corresponding tree of finite height. Then the following are equivalent:
\begin{itemize}
\item $\A$ is computably categorical.
\item $\A$ is relatively computably categorical.
\item $(T_{\A},\prec)$ is computably categorical.
\item $(T_{\A},\prec)$ is relatively computably categorical.
\item $(T_{\A},\prec)$ is of finite type.
\end{itemize}
\end{theorem}

We can characterize the ultrahomogeneous nested equivalence structures for  $n< \omega$ using this correspondence with trees. 

\begin{theorem} \label{uhn}
Let $\A=(A,E_1,\ldots,E_n)$ be a nested $n$-equivalence structure and let $E_0 = A \times A$ and $E_{n+1}$ be equality.  Then the following are equivalent. 
\begin{enumerate}
\item  $\A$ is ultrahomogeneous.
\item For  each $ i \leq n$ there exists $k_i$ such that every $E_i$ class is partitioned into $k_i$ many $E_{i+1}$ classes.
\item $T_{\A}$ is ultrahomogeneous in the predecessor representation. 
\end{enumerate}
\end{theorem}

\begin{proof}  $(1 \implies 2)$:  Suppose that $\A$ is ultrahomogenous but the second condition fails. Then for some $i \leq n$ there are two $E_i$ classes $C_1,C_2$ such that $C_1$ contains more $E_{i+1}$ classes than $C_2$. For $x \in C_1$ and $y\in C_2$ we have $\la x\ra\cong\la y\ra$ but the isomorphism can't extend since such an automorphism would have to send $C_1$ to $C_2$.

$(2 \implies 3)$: Assuming (2), it follows that each node $[a]_i$ of $T_{\A}$ has exactly $k_i$ immediate successors. Thus $T_{\A}$ is ultrahomogeneous by Theorem \ref{prop10}.

$(3 \implies 1)$:  Assume that $T_{\A}$ is ultrahomogeneous and let $\theta$ be an isomorphism mapping a finite subset $B$ of $A$ to a finite subset $C$. This induces an isomorphism $\widehat{\theta}$  mapping the finite subtree $\{[b]_i: i \leq n+1\}$ of $T_{\A}$ to $\{[c]_i: i \leq n+_1\}$ such that $\widehat{\theta}([b]_i) = [\theta(b)]_i$.  To check that this is an isomorphism, note first that, for all $b \in B$, $[b]_{n+1} = \{b\}$ and  $\widehat{\theta}([b]_{n+1}) = [\theta(b)]_{n+1} = \{\theta(b)\}$. Then for each $i$, $[b]_i$ is the predecessor of $[b]_{i+1}$ in $T_{\A}$ and $[\theta(b)]_i$ is the predecessor of $[\theta(b)]_{i+1}$, so that $\widehat{\theta}$ preserves the predecessor function on $T_{\A}$. Since $T_{\A}$ is ultrahomogeneous in the predecessor framework, $\widehat{\theta}$ may be extended to an automorphism $\widehat{\phi}$ of $T_{\A}$. Now define the automorphism $\phi$ on $\A$ so that, for each $a$, $\phi(a) = d$ if and only if $\widehat{\phi}(\{a\}) = \{d\}$, so that $\widehat{\phi}([a]_{n+1}) = [\phi(a)]_{n+1}$.  Since $[a]_n$ is the predecessor of $[a]_{n+1}$, it follows that $\widehat{\phi}([a]_n])$ is the predecessor of $\widehat{\phi}([a]_{n+1}) = [\phi(a)]_{n+1}$, which is $[\phi(a)]_n$. Proceeding by induction, we see that $\widehat{\phi}([a]_i) = [\phi(a)]_i$ for all $i$.  Then  for each $a,a' \in A$ and all $i \leq n+1$, we have 
\[
a E_i a' \iff [a]_i = [a']_i \iff \widehat{\phi}([a]_i) = \widehat{\phi}([a']_i) \iff [\phi(a)]_i = [\phi(a')]_i \iff \phi(a) E_i \phi(a')
\]
which shows that $\phi$ is an automorphism of $\A$. 
\end{proof}

\begin{cor} \label{cor91} If $\A = (A,E_1,\dots,E_n)$ is a nested ultrahomogeneous equivalence structure such that all  equivalence classes are finite, then $\A$ is ultrahomogeneous if and only if each $(A,E_i)$ is ultrahomogeneous. 
\end{cor}

\begin{proof} Suppose that each $(A,E_i)$ is ultrahomogeneous, so that all $E_i$ classes have the same finite size $m_i$. It follows that each $E_i$ class is partitioned into exactly $m_i/m_{i+1} = k_i$ many $E_{i+1}$ classes.  Conversely, if $\A$ is ultrahomogenous, then by Theorem \ref{uhn}, there exist $k_0,\dots,k_n$ such that each $E_i$ class is partitioned into $k_i$ many $E_{i+1}$ classes. For $i = n$, it follows that each $E_n$ class has size $k_n$. Then by induction, we see that for $i = 1,\dots,n$, each $E_i$ class has size $k_i \cdot k_{i+1} \cdots k_n$. Thus $(A,E_i)$ is ultrahomogeneous for $i = 1,\dots,n$. 
\end{proof}

Here is an example which shows that the finiteness condition is necessary in Corollary \ref{cor91}

\begin{example}
 Let $E_1$ be congruence modulo 2 over the natural numbers, which partitions $\omega$ into two  infinite classes, the odd numbers and the even numbers. Now  let $E_2$ partition the even numbers modulo 3 and partition the odd numbers modulo 5.
Then each $E_2$ class is also infinite, so that both $(A,E_1)$ and $(A,E_2)$ are ultrahomogeneous.  But $\A$ is not ultrahomogeneous, since $[0]_1$ is partitioned into 3 subclasses, whereas $[1]_1$ is partitioned into 5 subclasses. 
\end{example}

Unlike in the ultrahomogeneous case, it is not the case, even for a weakly homogeneous nested equivalence structure, that each individual equivalence relation must be weakly ultrahomogeneous.

\begin{example}
Let $\A=(A,E_1,E_2)$ where $E_1$ has two infinite classes $B$ and $C$, while $E_2$ partitions $B$ into two element classes and $C$ into three element classes.  So $(A,E_2)$ is not weakly ultrahomogeneous. Now let $S = \{b,c\}$ be our exceptional set with $b\in B,c\in C$ and suppose $\la b,c,\textbf{x}\ra\cong\la b,c,\textbf{y}\ra$.  Since the isomorphism respects the partition $\{B,C\}$ and both of $(B,E_2),(C,E_2)$ are ultrahomogeneous, the isomorphism extends on $B$ and $C$ separately, hence it extends to an automorphism of $\A$.
\end{example}

The above example indicates a property of 2-equivalence structures ensuring they are weakly ultrahomogeneous. This property leads to the following characterization, again using the connection with trees.

\begin{theorem} \label{wuhn}
Let $\A=(A,E_1,\ldots,E_n)$ be a nested $n$-equivalence structure and let $E_0 = A \times A$ and $E_{n+1}$ be equality.  Then $T_{\A}$ is weakly ultrahomogeneous in the predecessor representation if and only if $\A$ is weakly ultrahomogeneous. 
\end{theorem}

\begin{proof}  Suppose that $T_{\A}$ is weakly ultrahomogeneous with exceptional subtree $S$ and let  $C \subset A$ contain a representative $c$ for each leaf $x$ of $S$. 
Now let  $\theta$ be an isomorphism mapping a finite subset $B$ of $A$ to a finite subset $D$ with $\theta(c) = c$ for each $c \in C$. This induces an isomorphism $\widehat{\theta}$  mapping the finite subtree $\{[b]_i: b \in B, i \leq n+1\}$ of $T_{\A}$ to $\{[d]_i: d \in D,i \leq n+_1\}$ such that $\widehat{\theta}(x) = x$ for all $x \in S$ and such that $\widehat{\theta}([b]_i) = [\theta(b)]_i$. $\widehat{\theta}$ preserves  predecessors as seen in the proof of Theorem \ref{uhn} above. Since $T_{\A}$ is weakly ultrahomogeneous in the predecessor framework, with exceptional set $S$, it follows that $\widehat{\theta}$ may be extended to an automorphism $\widehat{\phi}$ of $T_{\A}$.  Then as in the proof of Theorem \ref{uhn}, we may extend $\theta$ to an automorphism $\phi$ on $\A$ so that, for each $a$, $\phi(a) = b$ if and only if $\widehat{\phi}(\{a\}) = \{b\}$, so that $\widehat{\phi}([a]_{n+1}) = [\phi(a)]_{n+1}$.

Suppose next that $\A$ is weakly homogeneous with exceptional set $C$ and let $S = \{[c]_i: c \in S, i \leq n\}$. Let $x = [c]_i \in S$. By Theorem \ref{pwuh}, it suffices to show that $T_S[x]$ is ultrahomogeneous. Suppose by way of contradiction that $T_S[x]$ is not ultrahomogeneous. Then there are extensions $[a]_j$ and $[b]_j$ of $x$ with $i < j$ and $h < k$ such that $[a]_j$ is partitioned into $h$ many $E_{j+1}$ classes and $[b]_j$ is partitioned into $k$ many classes. We claim that $S \cup \{a\} \cong S \cup \{b\}$. That is,  by the definition of $T_S[x]$, we have for any $d \in S$,   $a E_i d \iff c E_i d \iff b E_i d$ and we have 
$\neg  a E_{i+1} d$ and $\neg b E_{i+1} d$.  But this partial isomorphism cannot be extended to an automorphism since $h \neq k$. 
\end{proof}

For nested 2-equivalence structures, Theorem \ref{wuhn} and Proposition \ref{3wuht} lead to the following characterization. 

\begin{cor}\label{n2eq} Let $\A=(A,E_1,E_2)$ be a nested 2-equivalence structure. Then $\A$ is weakly ultrahomogeneous if and only if $\A$ satisfies the following conditions:  $E_1$ is weakly ultrahomogeneous, $E_2$ restricted to each $E_1$-class is weakly ultrahomogeneous, and there are $h,k\leq\omega$ such that all but finitely many of the restrictions are ultrahomogeneous with $h$ many 2-classes of size $k$.
\end{cor}

Since $n$-equivalence structures are relational, every weakly ultrahomogeneous $n$-equivalence structure is computably categorical.  But unlike for equivalence structures and linear orders, this containment is strict even for nested 2-equivalence structures.

\begin{example} Let $\A=(A,E_1,E_2)$ be a computable nested 2-equivalence structure where $E_2$ has infinitely many classes of size 2 and $E_1$ is such that infinitely many $E_2$-classes are split into $E_1$-classes of size 1 and infinitely many $E_2$-classes contain a single $E_1$-class of size 2. Given $x\in A$, find $y E_2x$. Asking if $y E_1 x$ will then let us computably determine which type of $E_2$-class $x$ belongs to. Then via a back and forth construction we see that $\A$ is computably categorical. But $\A$ is not weakly ultrahomogeneous since two types of restrictions appear infinitely often. 
\end{example}

\section{Conclusion}

In this paper, we have introduced the notion of a weakly ultrahomogeneous strucuture and  
examined several types of  countable ultrahomogeneous and weakly
ultrahomogeneous structures. We observed that countable
ultrahomogeneous linear orderings and also countable ultrahomogeneous 
equivalence structures all have computable models. We showed that for
computable linear orders and computable equivalence structures, the
weakly homogeneous structures are exactly the
computably categorical structures. For computable injection 
structures, we observed that there are continuum many ultrahomogeneous structures
with no computable copy, and there are ultrahomogeneous structures
which are not computably categorical, although every computably
categorical structure is weakly ultrahomogeneous. 
We proved that any computable weakly ultrahomogeneous structure is
$\Delta^0_2$ categorical .We also showed that there are continuum many ultrahomogeneous
trees $(T,f)$ under predecessor and that all are relatively computably categorical. 
We made a connection between trees under predecessor and nested equivalence structures. 

For these structures, we used index sets to determine the complexity of the property of being ultrahomogeneous, 
and the complexity of the property of being weakly ultrahomogeneous.

We introduced the notion of a \emph{minimal exceptional set} for
weakly ultrahomogeneous structures. We gave characterizations of these
sets for linear orders, equivalence structures, and injection
structures. We also looked at the set of elements definable from the
minimal exceptional set. 

Future research topics include the study of other structures such as
vector spaces, Boolean algebras, Abelian $p$-groups, and
partial orderings. One goal is to classify the weakly ultrahomogeneous
structures. A second goal is to determine the effective categoricity
of computable weakly ultrahomogeneous structures.


\end{document}